\documentclass[11pt]{article}
\usepackage{geometry}                
\usepackage{amsmath,amsthm,amssymb}
\usepackage{graphicx, blkarray, bigstrut, caption, multirow, tikz }
\usetikzlibrary{arrows,decorations.pathmorphing,shapes}
\newtheorem{theorem}{Theorem}

\newtheorem{lemma}{Lemma}
\newtheorem{corollary}{Corollary}
\newtheorem{observation}{Observation}

\theoremstyle{definition}
\newtheorem{definition}{Definition}
\newtheorem{question}{Question}
\newtheorem{cor}{Corollary}
\usepackage{graphicx, cite}

\usepackage{rotating}

\usepackage{blkarray}
\usepackage{tikz-feynman}
\usepackage{multirow}
\usepackage{mathrsfs}

\usepackage{bigstrut}

\usepackage{multirow}
\usepackage{tikz}
\usepackage{amsmath}
\usetikzlibrary{arrows,decorations.pathmorphing,shapes}

\usepackage{stmaryrd}
\newcommand{\ul}[1]{\underline{#1}}

\newcommand{\GS}{{S(G)}}

\newcommand{\SCR}{{\mathbb{CR}}}
\newcommand{\SCS}{{\mathbb{R}}}

\newcommand{\ocr}{o}

\newcommand{\LL}{{\mathcal{L}}}
\newcommand{\RR}{{\mathcal{R}}}
\newcommand{\TT}{{\mathcal{D}}}

\begin{document}
\title{ Cheating Robot Games: A model for insider information}
\begin{center}
\uppercase{\bf Cheating Robot Games: A model for insider information}
\vskip 20pt
{\bf Melissa A. Huggan\footnote{Supported by NSERC PDF-532564-2019 while the research took place.}}\\
\textit{\small{Department of Mathematics and Computer Science, Mount Allison University\\ 
Sackville, New Brunswick, Canada}}\\
{\tt mhuggan@mta.ca}\\ 
\vskip 10pt
{\bf Richard J. Nowakowski\footnote{Supported by NSERC grant 2019-04914.}}\\
\textit{\small{Department of Mathematics and Statistics, Dalhousie University\\
Halifax, Nova Scotia, Canada}}\\
{\tt r.nowakowski@dal.ca}\\ 
\end{center}

\begin{abstract} 
Combinatorial games are two-player games of pure strategy where the players,
usually called Left and Right, move alternately.  In this paper, we introduce Cheating Robot games.  
These arise from simultaneous-play combinatorial games where one player has insider information (`cheats'). Play occurs in rounds.
 At the beginning of a round, both players know the moves that are available to them. Left chooses a move. Knowing Left's move, Right then chooses a move. Right's move is not constrained by Left's choice. The round is not completed until both players have made a choice. A game is finished only when one or both players do not have a move at the beginning of a round. Right choosing a move, knowing Left's, makes the games deterministic, distinguishing them from simultaneous games. Also, the ending condition distinguishes this class of games from combinatorial games, since the outcomes are now Left-win, Right-win and draw. 
 
 The basic theory and properties are developed, including showing that there is an equivalence relation and partial order on the games. Whilst there are no inverses in the class of all games, we show that there is a sub-class, simple hot games, in which the integers have inverses.  In this sub-class, the optimal strategies are obtained by the solutions to a minimum-weight matching problem on a graph whose number of vertices equals the number of summands in the disjunctive sum.
\end{abstract}

\section{Introduction}

 The Cheating Robot model is inspired by two approaches to simultaneous games. One, and where the name originates, is the Japanese robot that always wins \emph{Rock}-\emph{Paper}-\emph{Scissors} against humans  (see~\cite{ItoSYI2016},  \cite{ShortHVS2010}). When played between humans, \emph{Rock}-\emph{Paper}-\emph{Scissors} is a simultaneous game. However, in Human-versus-Robot, the robot `cheats' by having reflexes fast enough to see what the human hand is forming and then responds correctly. This changes the simultaneous game into one in which the players play alternately but complete their moves in the same round. This is the same situation for simultaneous Human-versus-Human games where one player, by some means, has insider knowledge of what their opponent will play immediately before he plays and has time to react. 
 
  In this paper, we restrict the underlying games to the zero-sum combinatorial games presented in  \cite{AlberNW2019, BerleCG2001, Siege2013}. These are two-player games of perfect information and no chance devices, where the players move alternately. Examples include  \textsc{checkers}, \textsc{chess}, and \textsc{go}. The players are \textit{Left} and \textit{Right}. The usual winning convention is normal play, where the player who cannot move is the loser. The theory of alternating-play combinatorial games
  is particularly useful when applied to games that decompose into components. Our theory is also based on games which decompose. For example, consider the simple \textsc{toppling dominoes} position consisting of three components, 
  \[
\includegraphics[width=0.3\textwidth]{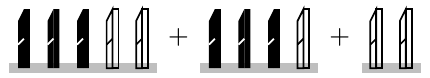}
\]
Left can topple a black domino and Right a white domino. Each move topples everything in the same direction in that component. 
Now we add the same-round, Cheating Robot aspect. As Left is choosing a domino to topple, the robot (Right) sees Left's choice and has time to react. When two dominoes are chosen in the
same component, if they topple towards each other then only the dominoes between them are toppled. In this example, if Right moves in the same component as Left, then after two rounds, the situation is
  \[
\includegraphics[width=0.3\textwidth]{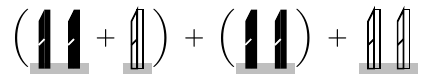}
\]
Left has four moves remaining and Right only three, so Left wins.  A little work shows that there is a Nash equilibrium: in the first round, Left should play in the first component, and Right in whichever of the first two components Left doesn't play. This results in,
  \[
\includegraphics[width=0.2\textwidth]{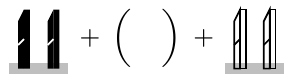}
\]
which is a draw. Players will lose if they deviate from their Nash strategies. Section \ref{subsec:toppling dominoes} illustrates how to find a Nash equilibrium for a class of \textsc{cheating robot toppling dominoes} positions.
  
More formally, in the Cheating Robot model, games have the following properties. There are two players, called \textit{Left} (female) and \textit{Right} (male). One player is the \textit{robot}; throughout this paper, it will be Right. Left is aware that Right is cheating and thus the games are deterministic. They play in rounds which consists of one move by each player, Left playing first. In a single round, Right's availability of moves is not affected by Left's choice. If one, or both, of the players cannot move, the game is over and a winner is declared according to some convention. In this paper, a player wins if they can move but their opponent cannot, and the game is a draw if neither has a move.

We start the development of the theory of combinatorial games played under the Cheating Robot model. There are several ways combinatorial games can be extended to Cheating Robot play. In keeping with combinatorial game theory, our guiding philosophy is twofold: (i) \textit{at the beginning of each round, the players must know when they have a move and what these are;} and (ii) \textit{if a player does not have a move they cannot win}. 
 
 The first part of the philosophy impacts the structure of the games. At the beginning of a same-round move, the players each have a set of available moves. Given our model, even though the robot (Right) can react to Left's move, his available moves are independent of her choice. This means, all pairs of moves are legal, and this must be reflected in redefining the alternating-play rules to same-round-play. Since the move sets are understood from the beginning of a same-round move, this allows us to use a matrix to define a position in terms of its moves.  The second part of the philosophy is the extension of the normal play winning convention to the Cheating Robot model.

 This is the approach  the authors of \cite{BahriK2010} suggested as a way of approximating values for simultaneous play in \textsc{domineering} \cite{CincoI}. Further, in~\cite{Huggan}, several combinatorial rulesets are extended to simultaneous rulesets which can then be reinterpreted in terms of Cheating Robot games (\emph{cr-games}). 

Combinatorial games, under normal play, have the following properties. They form an equivalence relation under a notion of `equality'. The equivalence classes have unique smallest games (smallest game tree) taken as the representative for the class called the \emph{canonical form}. There is a natural notion of addition and additive inverse. The  canonical forms are a partially ordered, abelian group. The main question is therefore, how much of the mathematical structure is carried over? The basic notions of disjunctive sum and equality do carry over, therefore
 some of this structure will be present in the Cheating Robot theory.

Our paper proceeds as follows. We develop the basic theory for all short (finite number of positions and no repeated positions) cr-games in Section \ref{sec:CR_theory}. First, we introduce the concepts of moves and options. 
Theorem \ref{thm:FTCR} proves there are only three outcome classes: Left-win, Right-win, and draw. The disjunctive sum induces an equivalence relation and a partial order on the games (see Theorem \ref{thm:equiv_po}). Theorems \ref{thm:dom2}, \ref{thm:Rightdom2}, and \ref{thm:dom1} provide simplifications of cr-games using the partial order. 
In Section~\ref{sec: simple hot games}, we investigate a class of cr-games that typically appear as endgame positions. Essentially, this class consists of combinations of games in which each game has one same-round move remaining in which the players can gain an advantage. To find the best order in which to play these positions, Left only needs a simple search, however, Right has to solve a quadratic-time optimization problem, Theorem \ref{thm:strategy2}. In Section \ref{subsec:toppling dominoes},  we give an example of solving the endgame of \textsc{cheating robot toppling dominoes}.


\section{Rules and Approaches for the Cheating Robot Model}\label{sec:CR_theory}

One important aspect of combinatorial games is that they frequently decompose into components. Players get to choose the component in which to play, independent of their opponent. In alternating play, as in \cite{AlberNW2019,BerleCG2001,Siege2013}, the components are analyzed first and 
this information is used to reduce the analysis needed for the total game. We use this approach in the Cheating Robot model. We copy, as much as is reasonable, the notation of normal play combinatorial games from \cite{AlberNW2019,BerleCG2001,Siege2013}.

Care is needed when defining the concepts that are seemingly clear when generalizing from alternating combinatorial game theory.
We differentiate between a move and an option. This is necessary. For example, in \textsc{domineering}, individually both Left and Right may be able to place a domino. If they overlap, then the same-round move cannot be the same as the Left move followed by the Right move. 
In a graph theoretic context, the moves correspond to the arcs and the options to the nodes.

Our philosophy, that both players must know when they have a move, demands that both Left and Right have a move if and only if there is a same-round move corresponding to both.  This is formalized in the next definition. Note that an \emph{out-arc} refers to an arc emanating from a vertex. For example, if $G$ has two vertices $x$ and $y$ and there is an arc from $x$ to $y$, then $x$ has an out-arc.

 \begin{definition}\label{def: ruleset_philosophy_1} 
 The \emph{game graph} of a cr-game is a directed multi-graph with coloured arcs. Each node is a \textit{position}, each arc is a \textit{move}.
 From a node, an out-arc is blue, red or black. Moreover,
 there is a bijection between the pairs of blue and red out-arcs and the black out-arcs. The blue out-arcs are the \textit{Left moves}, the red out-arcs are the \textit{Right moves},
and the black out-arcs are the  \textit{same-round moves} corresponding to the associated blue, red arcs. The head of a blue (red, black) out-arc is a Left (Right, same-round) \textit{option}. A \textit{follower} of a node $G$, is any node that can be reached by a sequence of out-arcs starting at $G$.
 \end{definition}
 
The blue and red arcs from Definition~\ref{def: ruleset_philosophy_1} are necessary to describe positions that have decomposed into components. 

To illustrate Definition~\ref{def: ruleset_philosophy_1}, consider \textsc{domineering}.  The game is played on a partial checkerboard, where Left places a $2\times1$ domino and Right, a $1\times 2$ domino. In the combinatorial game, overlaps are not allowed, but in the cr-game, the dominoes played on the same-round may overlap. See Figure~\ref{fig: game graph} for an example of a node and its out-arcs in the game graph of a \textsc{cheating robot domineering} position. 

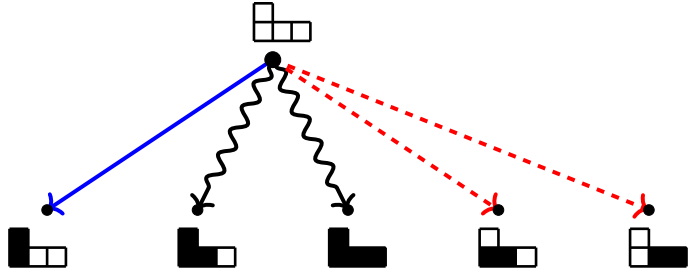
\begin{figure}[htb]
\begin{center}

\begin{tikzpicture}[decoration = {snake, pre length=3pt,post length=7pt,
                    }]
\draw [->, color = blue, line width=1.5pt] (0,0.)-- (-3,-2.);
\draw [decorate,->, color = black, line width=1.5pt] (0,0.)-- (-1.,-2.);
\draw [decorate,->, color = black, line width=1.5pt] (0,0.)-- (1.,-2.);
\draw [dashed, ->, color = red, line width=1.5pt] (0,0.)-- (3.,-2.);
\draw [dashed, ->, color = red, line width=1.5pt] (0,0.)-- (5.,-2.);

\draw [line width=1.pt] (-3.5,-2.25)-- (-3.25,-2.25);
\draw [line width=1.pt] (-3.5,-2.75)-- (-3.5,-2.25);
\draw [line width=1.pt] (-3.25,-2.75)-- (-3.25,-2.25);
\draw [line width=1.pt] (-3.5,-2.75)-- (-2.75,-2.75);
\draw [line width=1.pt] (-3.5,-2.5)-- (-2.75,-2.5);
\draw [line width=1.pt] (-3.,-2.5)-- (-3.,-2.75);
\draw [line width=1.pt] (-2.75,-2.5)-- (-2.75,-2.75);

\draw [line width=1.pt] (-1.25,-2.75)-- (-1.25,-2.25);
\draw [line width=1.pt] (-1.,-2.75)-- (-1.,-2.25);
\draw [line width=1.pt] (-1.25,-2.25)-- (-1.,-2.25);
\draw [line width=1.pt] (-1.25,-2.75)-- (-0.5,-2.75);
\draw [line width=1.pt] (-1.25,-2.5)-- (-0.5,-2.5);
\draw [line width=1.pt] (-0.5,-2.5)-- (-0.5,-2.75);
\draw [line width=1.pt] (-0.75,-2.5)-- (-0.75,-2.75);

\draw [line width=1.pt] (0.75,-2.75)-- (0.75,-2.25);
\draw [line width=1.pt] (1,-2.75)-- (1,-2.25);
\draw [line width=1.pt] (0.75,-2.25)-- (1,-2.25);
\draw [line width=1.pt] (0.75,-2.75)-- (1.5,-2.75);
\draw [line width=1.pt] (0.75,-2.5)-- (1.5,-2.5);
\draw [line width=1.pt] (1.5,-2.75)-- (1.5,-2.5);

\draw [line width=1.pt] (2.75,-2.75)-- (3.5,-2.75);
\draw [line width=1.pt] (2.75,-2.5)-- (3.5,-2.5);
\draw [line width=1.pt] (2.75,-2.25)-- (2.75,-2.75);
\draw [line width=1.pt] (3,-2.25)-- (3,-2.75);
\draw [line width=1.pt] (2.75,-2.25)-- (3,-2.25);
\draw [line width=1.pt] (3.5,-2.5)-- (3.5,-2.75);

\draw [line width=1.pt] (4.75,-2.75)-- (5.5,-2.75);
\draw [line width=1.pt] (4.75,-2.5)-- (5.5,-2.5);
\draw [line width=1.pt] (4.75,-2.25)-- (4.75,-2.75);
\draw [line width=1.pt] (5,-2.25)-- (5,-2.75);
\draw [line width=1.pt] (4.75,-2.25)-- (5,-2.25);
\draw [line width=1.pt] (5.5,-2.75)-- (5.5,-2.5);

\draw [line width=1.pt] (-0.25,0.25)-- (0.5, 0.25);
\draw [line width=1.pt] (-0.25,0.5)-- (0.5, 0.5);
\draw [line width=1.pt] (-0.25,0.25)-- (-0.25, 0.75);
\draw [line width=1.pt] (0,0.25)-- (0, 0.75);
\draw [line width=1.pt] (0.25,0.25)-- (0.25, 0.5);
\draw [line width=1.pt] (0.5,0.25)-- (0.5, 0.5);
\draw [line width=1.pt] (-0.25,0.75)-- (0, 0.75);
\begin{scriptsize}
\draw [fill=black] (0.,0.) circle (3.0pt);
\draw [fill=black] (-3.,-2.) circle (2.0pt);
\draw [fill=black] (-1.,-2.) circle (2.0pt);
\draw [fill=black] (1.,-2.) circle (2.0pt);
\draw [fill=black] (3.,-2.) circle (2.0pt);
\draw [fill=black] (5.,-2.) circle (2.0pt);

\path [fill=black] (-3.5,-2.25) rectangle (-3.25,-2.5);
\path [fill=black] (-3.5,-2.5) rectangle (-3.25,-2.75);

\path [fill=black] (-1.25,-2.25) rectangle (-1,-2.5);
\path [fill=black] (-1.25,-2.5) rectangle (-1,-2.75);
\path [fill=black] (-1,-2.5) rectangle (-0.75,-2.75);

\path [fill=black] (0.75,-2.25) rectangle (1,-2.5);
\path [fill=black] (0.75,-2.5) rectangle (1,-2.75);
\path [fill=black] (1,-2.5) rectangle (1.5,-2.75);

\path [fill=black] (2.75,-2.5) rectangle (3,-2.75);
\path [fill=black] (3,-2.5) rectangle (3.25,-2.75);

\path [fill=black] (5,-2.5) rectangle (5.25,-2.75);
\path [fill=black] (5.25,-2.5) rectangle (5.5,-2.75);
\end{scriptsize}
\end{tikzpicture}
\caption{Example of a \textsc{cheating robot domineering} position with one move for Left (solid, straight, blue arc) and two moves for Right (dashed, straight, red arcs). The squiggly black arcs represent the same-round moves. }\label{fig: game graph}
\end{center}
\end{figure}

In the rest of the paper, we will use the terms positions, options, and moves instead of nodes and arcs.
Note that Definition \ref{def: ruleset_philosophy_1} distinguishes between a `cr-game' and a `position'.  Following the imprecise usage in English, when it will not create ambiguity, we will be lax about applying the terms position and game (instead of cr-game).
 Typically, a node, $G$, might be referred to as a `game' instead of `position' when the followers of $G$ need to be considered. That is, the cr-game restricted to $G$ and the followers of $G$.

\begin{definition}\label{def: game}
Let $G$ be a position in a cr-game. Let $G(L)$ be the list of \emph{Left moves} and $G^{L}$ be the list of \emph{Left options}. A single Left move is denoted by $G(i,\cdot)$, with the corresponding single Left option denoted by $G^{i,\cdot}$. Similarly, let $G(R)$ be the list of \emph{Right moves}, and  $G^{R}$ the list of \emph{Right options}. A single Right move is denoted by $G(\cdot,j)$, with the corresponding single Right option denoted by $G^{\cdot,j}$. The associated \emph{same-round move} will be written as $G(i,j)$ and the option as $G^{i,j}$. 

The \textit{same-round matrix}, $S(G)$, is the matrix where the rows are indexed by the elements of $G(L)$ and the columns by the elements of 
$G(R)$. The $(i,j)$ entry of $S(G)$ is $G^{i,j}$. See Figure~\ref{fig: game matrix_S_G}.

  \end{definition}
 \begin{figure}[ht]
\begin{center}
$
\begin{matrix}
\\
\\
G(1,\cdot)\\
\\
 G(2,\cdot)\\
 \\
\vdots\\
\\
G(m,\cdot)\\

\\
\end{matrix}
\begin{matrix}
\begin{matrix}
&$$\quad G(\cdot,1)\quad$$&$$  G(\cdot,2)\quad $$&$$\ldots\quad \quad$$&$$ G(\cdot,n) \quad\quad$$\\
\end{matrix}\\
\left\llbracket
\begin{matrix}
&&&&\\
$$\,\,G^{{1,1}}\,\, $$&$$\,\,G^{{1,2}}\,\,&$$\,\,\ldots\,\,&$$\,\, G^{{1, n}}$$\\
&&&&&\\
$$G^{{2,1} }$$&$$\quad G^{{2,2}}\quad&$$\quad\ldots\quad &$$\,\, G^{{2,n}} $$\\
&&&&\\
$$\vdots & $$\vdots & $$\ddots & $$\vdots \\
&&&&\\
$$G^{{m,1}}$$ & $$G^{{m,2}} & $$\ldots & $$\,\, G^{{m,n}}$$ \\
&&&&&\\
\end{matrix}
\right\rrbracket\\
\end{matrix}
$
\caption{The matrix, $S(G)$, of a game $G$ with all the moves and options, where Left has $m$ moves and Right has $n$ moves.}\label{fig: game matrix_S_G}
\end{center}
\end{figure}

If only $G$ is being considered, then $S(G)$ contains all the needed information. However, when a game has split into components, players do not have to move in the same component. The information for a component must include the results of only one player playing in it. This extra information is $G^{L}$ and $G^{R}$. We expand $S(G)$ by adding an extra row and column, each indexed by $\emptyset$, representing that the player did not play. The entries of the row will be $G^R$ and the column $G^L$. The entry indexed by $(\emptyset,\emptyset)$ is the result of neither player playing, i.e., $G$ itself.
See Figure \ref{fig: game matrix_2}. We can refer to $G$ by $S^+(G)$ but $G=\{G^{L} \mid S(G) \mid G^{R}\}$ is frequently more useful, since analysis will
 need to consider the individual Left and Right options. 

\begin{figure}[ht]
\begin{center}
$
\begin{matrix}
\\
\\
G(1,\cdot)\\
\\
 G(2,\cdot)\\
 \\
\vdots\\
\\
G(m,\cdot)\\
\\
\emptyset\\
\\
\end{matrix}
\begin{matrix}
\begin{matrix}
&$$\quad G(\cdot,1)\quad$$&$$  G(\cdot,2)\quad $$&$$\ldots\quad \quad$$&$$ G(\cdot,n) \quad\quad$$&$$ \emptyset \quad\quad\quad$$\\
\end{matrix}\\
\left\llbracket
\begin{matrix}
&&&&&\\
$$\,\,G^{{1,1}}\,\, $$&$$\,\,G^{{1,2}}\,\,&$$\,\,\ldots\,\,&$$\,\,G^{{1, n}}\,\,&$$\,\,G^{1,\cdot}\,\,$$\\
&&&&&\\
$$G^{{2,1} }$$&$$\quad G^{{2,2}}\quad&$$\quad\ldots\quad &$$\quad G^{{2,n}}\quad &$$\quad G^{2,\cdot}\quad$$ \\
&&&&&\\
$$\vdots & $$\vdots & $$\ddots & $$\vdots & $$\vdots$$\\
&&&&&\\
$$G^{{m,1}}$$ & $$G^{{m,2}} & $$\ldots & $$G^{{m,n}} & $$G^{m,\cdot}$$ \\
&&&&&\\
$$G^{\cdot,1}$$ & $$G^{\cdot,2}$$ & $$\ldots $$& $$G^{\cdot,n}$$ & $$G$$ \\
&&&&&\\
\end{matrix}
\right\rrbracket\\
\end{matrix}
$
\caption{The matrix, $S^+(G)$, of a game $G$ with all the moves and options.}\label{fig: game matrix_2}
\end{center}
\end{figure}

It is tempting to regard the lists of options as sets because then duplications would be eliminated, as is the case in CGT. However, in any simplification the connection between the Left and the Right moves and the resulting same-round options must be maintained. 

Here we demonstrate Definition~\ref{def: game} using the position in Figure \ref{fig: game graph}. We have the following lists: 
 $G(L) =\{G(1,\cdot)\}$, $G(R)=\{G(\cdot,1), G(\cdot,2)\}$, $G^{L} =\{G^{1,\cdot}\}$,  $G^{R}=\{G^{\cdot,1}, G^{\cdot,2}\}$. The matrices of $S(G)$ and $S^+(G)$ are shown in Figure~\ref{fig: ex_mat}. The double bracket notation for $S(G)$ and $S^+(G)$ is a visual aid to remind readers that the choices are not made simultaneously.  
 
\begin{figure}
\begin{center}
\begin{minipage}[t]{0.37\textwidth}
$\begin{matrix}
\\
$$G(1,\cdot)$$\\
\end{matrix}
\begin{matrix}
\begin{matrix}
&\,\,$$G(\cdot,1)\,\,$$&&$$\,\,G(\cdot,2)\,\,$$\\
\end{matrix}\\
\left\llbracket
\begin{matrix}
$$\quad G^{1,1}\quad $$&&$$\quad G^{1,2}\quad $$ \\
\end{matrix}
\right\rrbracket,
\end{matrix}$
\end{minipage}
\begin{minipage}[t]{0.5\textwidth}
$
\begin{matrix}
\\
$$\quad G(1,\cdot)$$\\
$$\emptyset$$\\
\end{matrix}
\begin{matrix}
\begin{matrix}
$$\quad G(\cdot,1)\qquad$$&&$$G(\cdot,2)\,\,$$&&$$\qquad\emptyset\qquad$$\\
\end{matrix}\\
\left\llbracket
\begin{matrix}
&$$\quad G^{1,1}\quad $$&&$$\quad G^{1,2}\quad $$&&$$\qquad G^{1,\cdot}\quad$$ \\
&$$\quad G^{\cdot,1}\quad $$&&$$\quad G^{\cdot,2}\quad $$&&$$\quad G\quad$$ \\
\end{matrix}
\right\rrbracket.
\end{matrix}$
\end{minipage}
\end{center}
\caption{The matrices $S(G)$ (left) and $S^+(G)$ (right) of the position $G$ from Figure~\ref{fig: game graph}.}\label{fig: ex_mat}
\end{figure}

\begin{observation}
Note that the label on the Right moves is irrelevant. For example, if $A$ and $B$ are both Right moves, denoting $A$ by $G(\cdot, 1)$ and $B$ by $G(\cdot, 2)$ could have just as easily been denoted as $A$ by $G(\cdot, 2)$ and $B$ by $G(\cdot, 1)$. Similarly for Left. Observe that the effect of the relabeling on $S(G)$ is that the corresponding rows and/or columns will be permuted. For the game graph, this corresponds to permuting indices via  functions $\sigma$ for Left and $\pi$ for Right.  This implies a corresponding reordering on the same-round moves as well. For example, $G(i,\cdot)$, $G(\cdot,j)$, and $G(i,j)$ become 
$G(\sigma(i),\cdot)$, $G(\cdot,\pi(j))$, and $G(\sigma(i),\pi(j))$, respectively.
\end{observation}

 In terms of the notation from Definition~\ref{def: game}, $\{G^{L} \mid S(G) \mid G^{R}\}$ is unambiguous since changing the orders of the lists $G(L)$ and $G(R)$  only results in permuting $G^{L}$ and $G^{R}$ and also permuting
 the rows and columns of $S(G)$. This does not change the position.

A cr-game is \textit{short} if  it is has finitely many positions and there is no sequence of moves that repeat a position (loop-free). Here, we only consider short cr-games. The set of short cr-games can be defined recursively based on the depth of the game graph. The \emph{depth}, or maximum length, of a game graph is the longest directed path.\\

Note that if there is a dot as a placeholder for a list of options, this implies that the list is empty.
\begin{definition}\label{def:zero}  Let $0=\{\,\cdot\mid \cdot\mid\cdot\,\}$ and put \text{Day} $0=\{0\}$. The \textit{birthday} of $0$ is $0$. Let $G=\{G^{L} \mid S(G) \mid G^{R}\}$. If all options of $G$ have birthdays less than $n$, and there exists at least one option in \text{Day} $n-1$, then $G$ is in the set \text{Day} $n$ and its birthday, $b(G)$, is $n$. Let $\SCR=\cup_{n=0}^{\infty}\text{Day\,}\, n$.
\end{definition}

 Induction is a common tool used in proofs and is often referred to as `induction on the birthdays' or `induction on the options'. This is in lieu of saying `induction on the longest directed path of the game graph'.  The set of short cr-games, $\SCR$, consists of positions whose game graphs are of bounded length.

The aim of any competitive game is to win. In the alternating-play literature, there are three main winning conventions. 
For this paper, we only use the normal play convention: \textit{a player wins if they have moves and their opponent doesn't.} 
Cr-games have three possible outcomes, not the four of alternating play.

\begin{definition}\label{def:outcomes}$[\normalfont{\textrm{Outcomes}}]$   Let $G \in \SCR$. If $S(G)=\emptyset$, then Left wins if $G(L)\ne \emptyset$, denoted $\ocr(G)=\LL$; Right wins if $G(R)\ne \emptyset$, denoted $\ocr(G)=\RR$; otherwise $G$ is a draw, $\ocr(G)=\TT$. 
 If $S(G)\ne \emptyset$,  then $\ocr(G)=\mathcal{L}$,  if Left can force a win; $\ocr(G)=\mathcal{R}$, if Right can force a win; otherwise $G$ is a draw, $\ocr(G)=\mathcal{D}$. 

Alternately, we sometimes write
$G\in \LL$, $G\in\RR$, and $G\in \TT$, for a Left-win, Right-win, or draw, respectively.

\end{definition}

As is usual with combinatorial games, the outcome of a game can be evaluated by backtracking from the terminal positions of the game graph. This gives the Fundamental Theorem of cr-games. 

\begin{theorem}\label{thm:FTCR} Let $G \in \SCR$. The outcome class of $G$ is one of $\LL$, $\TT$, or $\RR$.
\end{theorem}

  \begin{proof} If $\GS=\emptyset$, then at most one of $G(L)$ and $G(R)$ is non-empty. From Definition \ref{def:outcomes}: if $G(L)=\emptyset=G(R)$ then $G$ is in $\TT$; if $G(L)\ne \emptyset$ then $G\in \LL$; 
  and if  $G(R)\ne \emptyset$, then $G\in\RR$.
   
   We may now assume that $\GS\ne \emptyset$. If, for some $G(i,\cdot)$, every $G^{i,j}$ is a Left win, then 
   $G\in \LL$ since Left can force a win.
     If, for each Left move, $G(i,\cdot)$, there exists a Right move $G(\cdot, j_i)$ such that the option $G^{i,j_i}$ is a Right win, then Right can force a win, i.e., $G\in \RR$. 
   Otherwise,  there exists $G(i,\cdot)$ where there is no $G(\cdot,j)$ such that $G^{i,j}$ is a Right win, but there is  $G(\cdot,k)$  such that $G^{i,k}$ is a draw, then $G \in \TT$.  The result now follows by induction on the options. 
  \end{proof}

Next, we formalize the definition of the disjunctive sum, $G+H$, in terms of moves and options.
We follow the standard CGT notation: 
If $R$ and $S$ are lists, then $R, S$ is the concatenation of the two lists;
if $B$ is a position and $A=A_1,A_2,\ldots,A_n$ is a list of positions, then $A+B$ is the list $A_1+B,A_2+B,\ldots,A_n+B$.

\begin{definition}\label{def: disjunctive sum} 
Let $G$ and $H$ be positions. For $G+H$,
the Left and Right moves are, respectively, $(G+H)(L)= G(L)+H, G+H(L)$ and $(G+H)(R)= G(R)+H, G+H(R)$;
 the Left and Right options are, respectively, $(G+H)^{L}=G^{L}+H,G+H^{L}$ and $(G+H)^{R}=G^{R}+H,G+H^{R}$.

The \emph{disjunctive sum} is written as
\begin{align*}
G+H &= \left\{ (G+H)^{L} \mid S(G+H)\mid (G+H)^{R}\right\}\\
&=\left\{ G^L+H, G+H^{L} \mid S(G+H)\mid G^R+H, G+H^R\right\},
\end{align*}
where $S(G+H)$ is obtained by Definition~\ref{def: game} and $S^+(G+H)$ is shown in Figure \ref{fig: game matrix_G+H}.
\end{definition}

\begin{sidewaysfigure}
$
\begin{matrix}
\\
\\
$$G(1,\cdot)+H$$\\
\\
$$ G(2,\cdot)+H$$\\
\\
$$\vdots$$\\
\\
$$G(m,\cdot)+H$$\\
\\
$$G+H(1,\cdot)$$\\
\\
$$G+H(2,\cdot)$$\\
\\
$$\vdots$$\\
\\
$$G+H(p,\cdot)$$\\
\\
$$\emptyset$$\\
\\
\end{matrix}
\begin{matrix}
\begin{matrix}
$$G(\cdot,1)+H$$&$$G(\cdot,2)+H$$&$$\ldots$$&$$ G(\cdot,n)+H\quad$$&$$ G+H(\cdot,1)$$&$$\quad G+H(\cdot,2)$$&$$\ldots$$&$$\quad G+ H(\cdot,q)\quad \quad$$&$$\quad\emptyset \quad \quad$$\\
\end{matrix}\\
\left\llbracket
\begin{matrix}
&&&&&&&&\\
$$G^{1,1}+H$$&$$ G^{1,2}+H$$&$$ \ldots$$&$$ G^{1,n}+H$$&$$ G^{1,\cdot}+H^{\cdot,1}$$&$$ G^{1,\cdot}+H^{\cdot,2}$$&$$ \ldots $$&$$\quad G^{1,\cdot}+H^{\cdot,q}$$&$$\quad G^{1,\cdot}+H \quad$$ \\
&&&&&&&&\\
$$G^{2,1}+H$$&$$G^{2,2}+H$$&$$\ldots$$&$$G^{2,n}+H$$&$$G^{2,\cdot}+H^{\cdot,1}$$&$$G^{2,\cdot}+H^{\cdot,2}$$&$$\ldots$$&$$G^{2,\cdot}+H^{\cdot,q}$$&$$G^{2,\cdot}+H$$\\
&&&&&&&&\\
$$\vdots$$&$$\vdots$$&$$\ddots$$&$$\vdots$$&$$\vdots$$&$$\vdots$$&$$\ddots$$&$$\vdots$$&$$\vdots$$\\
&&&&&&&&\\
$$G^{m,1}+H$$&$$G^{m,2}+H$$&$$\ldots$$&$$G^{m,n}+H$$&$$G^{m,\cdot}+H^{\cdot,1}$$&$$G^{m,\cdot}+H^{\cdot,2}$$&$$\,\,\ldots\,\,$$&$$G^{m,\cdot}+H^{\cdot,q}$$&$$G^{m,\cdot}+H$$\\
&&&&&&&&\\
$$G^{\cdot,1}+H^{1,\cdot}$$&$$G^{\cdot,2}+H^{1,\cdot}$$&$$\,\,\ldots\,\,$$&$$G^{\cdot,n}+H^{1,\cdot}$$&$$G+H^{1,1}$$&$$G+H^{1,2}$$&$$\ldots$$&$$G+H^{1, q}$$&$$G+H^{1, \cdot}$$ \\
&&&&&&&&\\
$$G^{\cdot,1}+H^{2,\cdot}$$&$$G^{\cdot,2}+H^{2,\cdot}$$&$$\,\,\ldots\,\,$$&$$G^{\cdot,n}+H^{2,\cdot}$$&$$G+H^{2,1}$$&$$G+H^{2,2}$$&$$\ldots$$&$$G+H^{2,q}$$&$$G+H^{2,\cdot}$$\\
&&&&&&&&\\
$$\vdots$$&$$\vdots$$&$$\ddots$$&$$\vdots$$&$$\vdots$$&$$\vdots$$&$$\ddots$$&$$\vdots$$&$$\vdots$$\\
&&&&&&&&\\
$$G^{\cdot,1}+H^{p,\cdot}$$&$$G^{\cdot,2}+H^{p,\cdot}$$&$$\,\,\ldots\,\,$$&$$G^{\cdot,n}+H^{p,\cdot}$$&$$G+H^{p,1}$$&$$G+H^{p,2}$$&$$\ldots$$&$$G+H^{p,q}$$&$$G+H^{p,\cdot}$$\\
&&&&&&&&\\
$$G^{\cdot,1}+H$$&$$G^{\cdot,2}+H$$&$$\ldots$$&$$G^{\cdot,n}+H$$&$$G+H^{\cdot,1}$$&$$G+H^{\cdot,2}$$&$$\ldots$$&$$G+H^{\cdot,q}$$&$$G+H$$\\
&&&&&&&&\\
\end{matrix}
\right\rrbracket\\
\end{matrix}
$
\caption{The matrix $S^+(G+H)$ which represents the disjunctive sum $G+H$, where Left has $m$ moves in $G$ and $p$ moves in $H$ and Right has $n$ moves in $G$ and $q$ moves in $H$.}\label{fig: game matrix_G+H}
\end{sidewaysfigure}

If, for all $X$, the outcome of $G+X$ is the same as the outcome for $H+X$, then there is no situation in which either player has a preference for $G$ or $H$. 
 In other words, both players are always indifferent to including either $G$ or $H$ in any sum. This indifference leads to the concepts of equality, and equivalence, which are defined relations. 
 
 Again, following the CGT convention, outcomes are totally ordered from Left's point of view, that is, $\LL > \TT > \RR$. Comparing outcomes leads to inequalities of games.

\begin{definition}\label{def: equality}
 Let $G$, $H \in \SCR$. 
 \begin{enumerate}
 \item  $G$ and $H$ are \textit{equal}, $G=H$,  if 
 $\ocr(G+X) = \ocr(H+X)$ for all $X\in \SCR$.
  \item $G$ is \textit{greater or equal} to $H$,  $G\succeq H$, if $\ocr(G+X)\geq  \ocr(H+X)$ for all $X\in \SCR$. 
  \item $G$ is \textit{greater} than $H$,  $G\succ H$, if $\ocr(G+X)\geq  \ocr(H+X)$ for all $X\in \SCR$ and $G\ne H$; $H\prec G$ if $G\succ H$.
  \item  Let $\SCS\subset \SCR$ then $G$ and $H$ are \textit{equivalent} modulo $\SCS$, $G\equiv_\SCS H$,
 if $\ocr(G+X) = \ocr(H+X)$ for all $X$ in $\SCS$.
 \end{enumerate}
\end{definition}

The next result shows that $0=\{\,\cdot\mid \cdot\mid\cdot\,\}$, (Definition \ref{def:zero}), is appropriately named.
\begin{lemma}\label{lem:zero}
For any position $G$, $G+0=G$.
\end{lemma}
\begin{proof}
By definition of disjunctive sum we have
\[
G+0 = \left\{(G+0)^{L} \mid S(G+0) \mid (G+0)^{R}\right\}\]
but since 0 has no moves $S(G+0)=S(G)$, and hence 
\begin{align*}
G+0=& \left\{G^{L}+0 \mid S(G)\mid G^{R}+0\right\}\\
=& \left\{G^{L} \mid S(G) \mid G^{R}\right\},
\end{align*}
and the last equality follows by induction. 
\end{proof}

If a game is comparable with $0$, then we know something about its outcome.
\begin{lemma}
Let $G \in \SCR$. If $G \succ0$, then $o(G)\geq \TT$. If $G\prec 0$, then $o(G) \leq \TT$.
\end{lemma}
\begin{proof}
Consider $G$, where $G\succ 0$. By the definition, $o(G+X)\geq o(0+X)$ for all $X$.  If we let $X$ be $0$, then
$o(G+0)\geq o(0+0)$. Clearly, $o(0+0)=\TT$, and since neither player has a move in $0$, then $o(G+0)=o(G)$.
Thus, if $G\succ 0$, then $\ocr(G) \geq \mathcal{D}$. Similarly, if $G\prec 0$, then $\ocr(G) \leq \mathcal{D}$.
\end{proof}
The converses are not true. Theorem \ref{thm: structure of G} gives a characterization for  $G\succ 0$.\\

In general, to show $G\succeq H$, we need to show both that:
(1) if Right can win $G+X$, then he can win $H+X$; and
(2) if Right can force a draw for $G+X$, then he can either force a draw or win $H+X$. 
To prove $G=H$ it suffices to show that $G\succeq H$ and $H\succeq G$.

The following is a typical approach
to find the outcome of $H+X$.  Start with a move,  $H(j,\cdot)+X$. If the order relation between some $G(i,\cdot)$ and $H(j,\cdot)$ is known, then use Right's known response to $G(i,\cdot)+X$ as a guide to Right's response in $H(j,\cdot)+X$. This approach is used often in CGT. \\

\begin{theorem}\label{thm:equiv_po}
 The equality relation  is an equivalence relation on $\SCR$. The relation $\succeq$ is a partial order on $\SCR/=$.
\end{theorem}

\begin{proof} For any position $G$, let $G \in \SCR$, clearly $G= G$. We know $\ocr(G+X) = \ocr(H+X)$ for all $X \in \SCR$ if and only if $ \ocr(H+X)=\ocr(G+X)$ for all $X \in \SCR$, and hence $G = H$ if and only if $H = G$. 

Let $G,H,K \in \SCR$.  If $\ocr(G+X) = \ocr(H+X)$ and $\ocr(H+X)=\ocr(K+X)$ for all $X\in \SCR$, then $\ocr(G+X) = \ocr(K+X)$. Thus $G= K$ and it follows that $=$ is an equivalence relation.

Next we prove that the relation $\succeq$ is a partial order on $\SCR/=$. Clearly, $G\succeq G$. If $\ocr(G+X) \geq \ocr(H+X)$ and $\ocr(H+X) \geq \ocr(G+X)$ for all $X\in \SCR$, then $\ocr(G+X) =\ocr(H+X)$ for all $X\in \SCR$ and hence, $G = H$ by definition. Lastly, for every $X\in\SCR$, if $\ocr(G+X) \geq \ocr(H+X)$ and $\ocr(H+X)\geq \ocr(K+X)$, then $\ocr(G+X) \geq \ocr(K+X)$. Thus $G\succeq K$ and it follows that $\succeq$ is a partial order.\hfill 
\end{proof}

\begin{theorem}\label{thm: properties} Let $G,H,K \in \SCR$. We have 
(1) $G+H= H+G$ and (2) $(G+H)+K = G+(H+K)$. 
\end{theorem}
\begin{proof}

(1) By definition, 
\begin{align}
G+H &= \left\{ (G+H)^{L} \mid S(G+H)\mid (G+H)^{R}\right\}\label{eq: definition}\\
&=\left\{ G^L+H, G+H^L \mid S(G+H)\mid G^R+H, G+H^R \right\}\\
&=\left\{ H+G^L, H^L+G \mid S(H+G)\mid H+G^R, H^R +G\right\}\label{eq: induction}\\
&=H+G,
\end{align}
where Equation~\ref{eq: induction} follows by induction and the fact that permuting the rows and columns of the matrix leaves the resulting game unchanged. Hence, $G+H = H+G$.

(2) To show $(G+H)+K=G+(H+K)$, we consider Definition~\ref{def: disjunctive sum} and each of $((G+H)+K)^L$,   $((G+H)+K)^R$, and $S((G+H)+K)$ in turn.  
\begin{align*}
((G+H)+K)^L&=(G+H)^L+K, (G+H)+K^L\\
&=(G^L+H)+K, (G+H^L)+K,  (G+H)+K^L\\
&=G^L+(H+K), G+(H^L+K),  G+(H+K^L)\quad\text{(by induction)}\\
&=G^L+(H+K), G+(H+K)^L\\
&=(G+(H+K))^L.
\end{align*}

\begin{align*}((G+H)+K)^R&=(G+H)^R+K, (G+H)+K^R\\
&=(G^R+H)+K, (G+H^R)+K,  (G+H)+K^R\\
&=G^R+(H+K), G+(H^R+K),  G+(H+K^R)\quad\text{(by induction)}\\
&=G^R+(H+K), G+(H+K)^R\\
&=(G+(H+K))^R.
\end{align*}
The same-round options follow from the previous two equalities. Examining the matrices we have:

\begin{center}
$
\begin{matrix}
\\
$$(G+H)(L)+K$$\\
\phantom{stuff}\\
$$(G+H)+K(L)$$\\
\end{matrix}
\begin{matrix}
\begin{matrix}
$$\quad (G+H)(R)+K\quad$$&$$ (G+H)+K(R)\quad $$\\
\end{matrix}\\
     \left\llbracket
     \begin{matrix}
$$S(G+H)+K\quad$$&$$(G+H)^{L}+K^{R}$$\\
&\\
$$(G+H)^R+K^L\quad$$&$$(G+H)+S(K)$$\\
\end{matrix}
\right\rrbracket
\end{matrix},$
\end{center}

and 
\begin{center}
$
\begin{matrix}
\\
$$G(L)+(H+K)$$\\
\phantom{stuff}\\
$$G+(H+K)(L)$$\\
\end{matrix}
\begin{matrix}
\begin{matrix}
$$\quad G(R)+(H+K)\quad $$&$$ G+(H+K)(R)\quad$$\\
\end{matrix}\\
     \left\llbracket
     \begin{matrix}
$$S(G)+(H+K)\quad $$&$$G^L+(H+K)^R$$\\
&\\
$$G^R+(H+K)^L\quad $$&$$G+S(H+K)$$\\
\end{matrix}
\right\rrbracket
\end{matrix}.$
\end{center}

Expanding out the moves, gives that the indexing of the rows and columns respectively are identical in both matrices, up to permutation. In particular, in both cases, we have that 
\begin{center}
$
\begin{matrix}
\\
$$G(L)+H+K$$\\
\phantom{stuff}\\
$$G+H(L)+K$$\\
\phantom{stuff}\\
$$G+H+K(L)$$\\
\end{matrix}
\begin{matrix}
\begin{matrix}
$$\quad G(R)+H+K$$&$$\quad G+H(R)+K\quad$$&$$ G+H+K(R)\quad$$\\
\end{matrix}\\
     \left\llbracket
     \begin{matrix}
$$S(G)+H+K\quad$$&$$G^{L}+H^{R}+K\quad$$&$$G^{L}+H+K^{R}$$\\
&&\\
$$G^{R}+H^{L}+K\quad$$&$$G+S(H)+K\quad$$&$$G+H^{L}+K^{R}$$\\
&&\\
$$G^{R}+H+K^{L}\quad$$&$$G+H^{R}+K^{L}\quad$$&$$G+H+S(K)$$\\
\end{matrix}
\right\rrbracket
\end{matrix}.$
\end{center}

 Note that $G(L)+H+K$ is shorthand for the Left moves in $G$ summed with the other components. Similarly for the other matrix indices. Hence \[S((G+H)+K) =S(G+(H+K)).\] Thus, $(G+H)+K= G+(H+K)$.
\end{proof}

\begin{theorem}\label{thm:order preserving}
Let $A,B,C,D\in \SCR$.
\begin{enumerate}
\item If $A\succeq B$, then $A+C\succeq B+C$;
\item If $A\succeq B$ and $C\succeq D$, then $A+C\succeq B+D$.
\end{enumerate}
\end{theorem}
\begin{proof}
By definition, if $A\succeq B$ then for all $Y \in \SCR$, $\ocr(A+Y)\geq \ocr(B+Y)$. Letting $Y$ be the game $C+X$, we obtain $\ocr(A+C+X)\geq \ocr(B+C+X)$ which means $A+C\succeq B+C$.

 In a similar fashion, letting  $Y$ be the game $B+X$ in the inequality $\ocr(C+Y)\geq \ocr(D+Y)$ 
gives $\ocr(C+B+X)\geq \ocr(D+B+X)$. Combining inequalities, and by commutativity of Theorem~\ref{thm: properties}, we have
 $\ocr(A+C+X)\geq \ocr(B+D+X)$ and the result follows. 
\end{proof}

The partial order can be used to identify poor options and, since we are assuming best play, these can then be eliminated, as shown in the next three results. The first two are, respectively, what Left can eliminate and what Right can eliminate.  The third, Theorem~\ref{thm:dom1}, shows that given a list of same-round options associated with a single Left move, Right can replace all of the dominated options by his best option. 

\begin{theorem}\label{thm:dom2}
 For $G \in \SCR$, suppose that, for some $i$ and $j$, 
\begin{enumerate}
\item  $G^{i, \cdot}\succeq G^{j, \cdot}$; and
\item  for all $r$ there exists  $s$ with $G^{i,r}\succeq G^{j,s}$.
\end{enumerate}

Let $H$ be the game $G$ without Left's move $G(j,\cdot)$,  then $H=G$.
\end{theorem}

\begin{proof}
Note that  $S^+(H)$ is $S^+(G)$ but without the row labeled $G(j, \cdot)$. Left's moves in $H$ are a subset of the moves in $G$, and Right's moves are the same in both games. This implies $G \succeq H$.   

Now we show $H\succeq G$. Let $X\in\SCR$. Suppose Right can win (draw) $H+X$.  In $G+X$, Right mimics the winning (drawing) moves 
from $H+X$. The exception is if Left moves to $G^{j, \cdot}+X$. Right now considers Left moving 
from $G+X$ to  $G^{i, \cdot}+X$. If the winning (drawing) Right move is to $G^{i, \cdot}+X^{\cdot, r}$ then in $G^{j, \cdot}+X$
Right moves to $G^{j, \cdot}+X^{\cdot, r}$. Since $G^{i, \cdot}\succeq G^{j, \cdot}$, then Right wins (at least draws) $G^{j, \cdot}+X^{\cdot, r}$. If instead, a winning move is to $G^{i,r}+X$ then there is some $s$ with 
$G^{i,r}\succeq G^{j,s}$
and Right moves from $G(j, \cdot)+X$ to $G^{j,s}+X$ and wins (or at least draws). Therefore $H\succeq G$ thus proving
$H= G$.
\end{proof}

\begin{theorem}\label{thm:Rightdom2}  Let $G = \{G^{L} \mid \GS \mid G^{R}\}$, with $m$ Left moves and $n$ Right moves.

Let $H$ be the game obtained from $G$ by deleting the move $G(\cdot,n)$. In $G$, if for every  $i$, there exists $j_i\ne n$ such that $G^{i,n}\succeq G^{i,j_i}$, then $G= H$.
\end{theorem}

\begin{proof} Since Right has more options in $G$ than in $H$ and Left has all the same options in both games, we have $H \succeq G$. 

To show $G \succeq H$, let $X\in \SCR$ and suppose that Right can win (or draw) $G+X$. We claim that Right can win (or at least draw) $H+X$. There is only one case where the moves in $H+X$ cannot mimic those in $G+X$. That is, consider for a Left move  $H(i, \cdot)+X$, this is also a Left move in $G+X$, and a winning Right response is $G^{i,n}+X$. In $H(i, \cdot)+X$, Right moves to $G^{i,j_i}+X$. By assumption,  $G^{i,j_i}\preceq G^{i,n}$, so since Right can win (or draw) $G^{i,n} + X$, Right can win (at least draw) $G^{i,j_i}+X$, which is identical to $H^{i,j_i}+X $. That is, Right can win (at least draw) $H + X$. 
\end{proof}

\begin{theorem}\label{thm:dom1}  For a game $G$, 
suppose $G^{i,1}\succeq G^{i,2}$. Let $G'$ be the game where $S^+(G')=S^+(G)$ except $G^{i,1}$ is replaced by $G^{i,2}$. Then $G= G'$.
\end{theorem}

\begin{proof}
 In $G$, Left has the same moves and Right has the same plus more moves than in $G'$, thus $G'\succeq G$. To show $G\succeq G'$, let $X\in \SCR$ and suppose that Right can win (draw) $G+X$. In $G'+X$, Right mimics the winning (drawing) moves from $G+X$. The exception is if Left moves to $G^{i,\cdot}$ and Right's winning (drawing) move in $G+X$ is to $G^{i,1}+X$. In $G'+X$, Right now responds to $G^{i,2}+X$ and since $G^{i,1}\succeq G^{i,2}$ Right wins (at least draws) $G^{i,2}+X$. Thus $G= G'$.
\end{proof}

\begin{definition}$[\normalfont{\textrm{Integers}}]$  Set $\ul{0}=\{\,\cdot\mid \cdot\mid\cdot\,\}$. For $n>0$, let $\ul{n}$ be the game $\left\{\ul{n-1} \mid \cdot \mid \cdot\,\right\}$; for $n<0$, let $\ul{n}$ be the game $ \left\{\,\cdot \mid \cdot \mid \ul{n+1}\right\}$. \end{definition}

For $n>0$, $\ul{n}$ is the game in which Left has $n$ moves and Right has none. Also, for brevity, we write $\left\{\ul{r} \mid \ul{s} \mid \ul{t}\right\}$ for the position where the same-round move is to $\ul{s}$.

\begin{theorem}\label{thm: structure of G} For a position $G$,  
 $G\succ 0$ if and only if $G(L)\ne\emptyset$ and $G(R) = \emptyset$.
\end{theorem}

\begin{proof}
Suppose that $G \succ 0$.  

Suppose $G(R)\ne \emptyset$. Let $b(G) = n$, meaning the depth of the game graph is $n$, and let $X = \left\{\underline{-n-1} \mid  \underline{0} \mid \underline{-n-1}\right\}$. We will show that $\ocr(G+X)<  \ocr(X)$, contrary to Definition~\ref{def: equality}. Since $\ocr(X)= D$, we need to show that Right wins $G+X$. 

Consider Left's moves:

(i) Left plays in $G$, Right plays in $X$ and wins (since he plays through his $n+1$ moves and Left has at most $n$ moves); or

(ii) Left plays in $X$, Right plays in $G$ and wins (since $b(G^{\cdot,j}) < n$ and Right has $n+1$ moves in $X^{1,\cdot}$). 

By assumption, at least one of $G(L)$, $G(R)$ is non-empty. Hence we must have $G(R)=\emptyset$, and $G(L) \neq \emptyset$.\\

Now suppose $G(L)\ne\emptyset$ and $G(R)= \emptyset$. In this case, $o(G)=\LL$, and thus $G \not\preceq 0$. It now suffices to show that if Right can win (draw) $G+X$ then Right can win (at least draw) $X$ for all $X \in \SCR$. 
 
Left playing in $0+X$ means Left plays in $X$; i.e., $X(i,\cdot)$, for some $i$. In $G+X$, with Left playing the move 
$G+X(i,\cdot)$, since $G(R) = \emptyset$, Right responds to $G+X^{i,j}$, for some $j$. Since Right can win (draw) $G+X^{i,j}$, he can win (at least draw) $X^{i,j}$ by induction.
\end{proof}

From Theorem~\ref{thm: structure of G}, it follows that if $G(R) \neq \emptyset$ then $G \not= 0$ and if $G(R) = \emptyset$ and $G(L) \neq \emptyset$ then again $G \not= 0$. That is, the equivalence class of $0$ contains only $0$. A further implication is that if there is a move in either $G$ or $H$, then $G+H\not= 0$.

\begin{cor}\label{cor:0unique}
If $G= 0$, then $G$ is identically $0$. Moreover, if $G+H= 0$, then $G=H=0$.
\end{cor}

 A game $G$ is a \textit{dicot}\footnote{In normal play these are called all-small but the term refers to the fact that the value of the game is an infinitesimal. This is not true under other winning conventions, and is unknown if they are infinitesimal in cr-games. The term `dicotic' was first coined for this class of games but in the literature authors have shortened it to dicot.} if, in any position, either both players have a move or the game is over \cite{Siege2013}. Dicot games are examples of cases where $0\succ G$ does not mean that $\ocr(G) = \mathcal{R}$.

\begin{lemma}\label{lem:dicot} Let $G\in\SCR$ be a  dicot game and $G\ne 0$, then $\ocr(G)=\TT$ and $0\succ G$.
\end{lemma}
\begin{proof} Let $G\in \SCR$ be a dicot game.
For any $X \in \SCR$, if Right wins or draws $X$,
then by replying in the same summand of $G+X$, Right will always have a response to Left's move, thus $\ocr(G+X)\leq \TT$.
 In addition, if Right wins $X$, then in $G+X$, eventually Right will have a move in a follower of $X$ but Left will not. Therefore, $\ocr(G+X)=\RR$ and thus $0\succ G$.
\end{proof}

\begin{theorem}\label{thm:int_ineq} Let $t,j,k$ be positive integers. 
\begin{enumerate}
\item $\ul{j}+\ul{k} = \ul{j+k}$ and the same is true if $j$ and $k$ are negative integers. 
\item  $\ul{k} \succ \ul{k-1}$. Also $\ul{-k+1}\succ \ul{-k}$; 
\item $\{\ul{k-1}\mid \ul{k}\mid \ul{k+t}\}= \ul{k}$ and $\{\ul{-k-t}\mid \ul{-k}\mid \ul{-k+1}\}= \ul{-k}$
\end{enumerate}
\end{theorem}

\begin{proof}  (1) In the next equations,
the second equality is from the definition of disjunctive sum and the third by induction and domination (Theorem \ref{thm:dom2}).

\begin{eqnarray*}
\ul{j}+\ul{k}&=&\{\ul{j-1}\mid  \cdot\mid  \cdot\,\}+\{\ul{k-1}\mid \cdot\mid \cdot\,\}\\
&=&\{\ul{j-1}+ \ul{k}, \ul{j}+ \ul{k-1}\mid  \cdot\mid  \cdot\,\}\\
&=&\{\ul{j+k-1}\mid  \cdot\mid  \cdot\,\}\\
&=&\ul{j+k}
\end{eqnarray*}

A analogous argument holds if $j$ and $k$ are negative.\\

(2) Let $X \in \SCR$. In $\ul{-1}+X$, Right has the same moves as in $X$ plus the extra move of playing in $\ul{-1}$, whereas Left has the same moves. 
Therefore $o(X)\geq o(\ul{-1}+X)$.

If Right can win (draw) $\ul{1}+X$, we claim Right can win (at least draw) $X$. Suppose in $X$,  Left 
plays $X(i,\cdot)$. Now in $\ul{1}+X$, with Left playing the move $\ul{1}+X(i,\cdot)$, Right has a winning (drawing) move to the option $\ul{1}+X^{i,j}$ for some $j$,
and, by induction, Right can win (at least draw) $X^{i,j}$.

Let $k>1$. In $\ul{-k}+X$ Right has one more move than in $\ul{-k+1}+X$ and Left has the same moves.  Therefore, $o(\ul{-k+1}+X)\geq o(\ul{-k}+X)$.

 Suppose Right can win (draw) $\ul{k}+X$. In response to Left's move to $\ul{k-1}+X$, Right has a winning (drawing) 
same-round response to $\ul{k-1}+X^{\cdot,j}$, for some $j$. 
Now in $\ul{k-1}+X$, if Left moves to $\ul{k-2}+X$ then Right moves to $\ul{k-2}+X^{\cdot,j}$. By induction, 
$\ul{k-1}\succ\ul{k-2}$ and so Right wins (at least draws) $\ul{k-2}+X^{\cdot,j}$. If Left's move is $\ul{k-1}+X(i,\cdot)$, then in
 $\ul{k}+X(i,\cdot)$ Right has a winning (drawing) move to $\ul{k}+X^{i,j}$ for some $j$. By induction Right wins (at least draws) $\ul{k-1}+X^{i,j}$ if he wins 
 (draws) $\ul{k}+X^{i,j}$. Therefore, $o(\ul{k}+X^{i,j})\geq o(\ul{k-1}+X^{i,j})$

Note, $\ul{k}\not= \ul{k-1}$ since $\ul{k}-\ul{k}$ is a draw but $\ul{k-1}-\ul{k}$ is a Right win.\\

(3) We want to show $\{\ul{k-1}\mid \ul{k}\mid \ul{k+t}\}= \ul{k}$. Fix $k,t>0$. Let $G=\{\ul{k-1}\mid \ul{k}\mid \ul{k+t}\}$.
We first show that $G\succeq \ul{k}$. Suppose Right can win (draw) $G+X$. Now, in $ \ul{k}+X$, Left can move in either component:
\begin{enumerate}
\item[i.] Suppose Left's move is $\ul{k}(1,\cdot)+X$. Now in $G+X$, Right considers that Left has played the move $G(1,\cdot)+X$ and Right has winning (at least drawing) responses in this round of moves. 
If one is 
$G(1,\cdot)+X(\cdot,j)$ giving the option $\ul{k-1}+X^{\cdot,j}$, then in $\ul{k}(1,\cdot)+X$ there is an option $\ul{k-1}+X^{\cdot,j}$ for Right, which he wins (at least draws). If there are no such winning (drawing) moves then, in $G(1,\cdot)+X$ Right wins (at least draws) by playing the move $G(1,1)+X$, that is, Right wins (at least draws) $\ul{k}+X$.
\item[ii.] Suppose Left's move is $\ul{k}+X(i,\cdot)$. By assumption, Right has a winning (drawing) response in $G+X(i,\cdot)$. If this is playing to $G+X^{i,j}$, then,
by induction, Right wins (at least draws) $\ul{k}+X^{i,j}$ since he wins (draws) $G+X^{i,j}$. Suppose, instead, the winning (drawing) move is  $G(\cdot,1)+X(i,\cdot)$, that is, to the option
$\ul{k+t}+X^{i,\cdot}$. However, $\ul{k+t}+X^{i,\cdot} \succ  \ul{k}+X^{i,\cdot}$ and thus, since Right can win (draw) $\ul{k+t}+X^{i,\cdot}$, then Right can win (at least draw) $\ul{k}+X^{i,\cdot}$.
\end{enumerate}

\noindent
Now we show that $G\preceq \ul{k}$. Suppose Right can win (draw) $\ul{k}+X$. In $G+X$, Left can move in either summand.
 \begin{enumerate}
\item[i.] Suppose Left's move is $G(1,\cdot)+X$. First note that $\ul{k-1}+X\preceq \ul{k}+X$. Since  Right wins (draws) $\ul{k}+X$, by assumption, then in $\ul{k-1}+X$ Right has a winning (at least drawing) move to $\ul{k-1}+X^{\cdot,j}$. Now, in $G(1,\cdot)+X$, Right can also move to 
$G^{1,\cdot}+X^{\cdot,j}$, and therefore Right wins (at least draws).
\item[ii.]  Suppose Left's move is $G+X(i,\cdot)$.  In $\ul{k}+X(i,\cdot)$, Right has a winning (drawing) move to $\ul{k}+X^{i,j}$. Now, by induction, Right wins (at least draws) $G+X^{i,j}$.
\end{enumerate}
\noindent
To prove the second part of $(3)$,  let $G=\{\ul{-k-i}\mid \ul{-k}\mid \ul{-k+1}\}$. To show $G\preceq \ul{-k}$, suppose Right can win (draw) $\ul{-k}+X$. In $G+X$, Left can move in either component.
 \begin{enumerate}
\item[i.] Suppose Left's move is $G(1,\cdot)+X$. Here Right moves  to $\ul{-k}+X$ and wins (at least draws) by assumption. 
\item[ii.] Suppose Left's move is $G+X(i,\cdot)$.  If Right wins  (draws) $\ul{-k}+X(i,\cdot)$ by moving to $\ul{-k}+X^{i,j}$ then Right wins  (at least draws) 
$G+X(i,\cdot)$, by induction, by playing to $G+X^{i,j}$.  If Right wins  (draws) $\ul{-k}+X(i,\cdot)$ by moving to $\ul{-k+1}+X^{i,\cdot}$, he wins  (at least draws) 
$G+X(i,\cdot)$ by moving in $G$ to
$\ul{-k+1}+X^{i,\cdot}$.
\end{enumerate}

 To show $G\succeq \ul{-k}$, suppose Right can win  (draw) $G+X$. In $\ul{-k}+X$, Left can only move in $X$, to $\ul{-k}+X(i,\cdot)$. If Right wins  (draws) $G+X(i,\cdot)$ by moving to $G+X^{i,j}$ then Right wins  (at least draws) $\ul{-k}+X^{i,j}$ by induction. If Right wins  (draws) $G+X(i,\cdot)$ by moving to $-\ul{k+1}+X^{i,\cdot}$, then he wins  (at least draws) $\ul{-k}+X(i,\cdot)$ by also moving to $\ul{-k+1}+X^{i,\cdot}$.
\end{proof}

\begin{observation}
The $\SCR$ games born by Day 
$1$ are: $0 = \{\cdot\mid\cdot \mid\cdot\}$, $\ul{1}=\{0\mid \cdot\mid\cdot\}$, 
$\ul{-1} = \{\cdot \mid \cdot \mid 0\}$, and $\overline{*} = \{0\mid 0\mid 0\}$. This list is exhaustive as Day 
 $1$ games can only have options born by Day 
  $0$ and so options are either empty or 0. 
  
The  partial order on these games is $1 \succ 0 \succ \overline{*} \succ -1$. 
Theorem \ref{thm:int_ineq} shows that $\ul{1} \succ 0 \succ\ul{-1}$. Lemma \ref{lem:dicot} shows that $0\succ \overline{*}$. 
Suppose Right can win (draw) $\overline{*} +X$, then Right can win (at least draw) $\ul{-1}+X$ by playing the corresponding moves.
The exception is if the winning (drawing) response in $\overline{*} +X(1,\cdot)$ is to $0+X^{1,\cdot}$ but then in $\ul{-1}+X^{1,\cdot}$
Right also moves to $0+X^{1,\cdot}$ and wins (at least draws).
\end{observation}

In CGT, the intuitive notion of the negative in normal play or `conjugate' in mis\`ere play, is to turn the board around thereby interchanging the moves, options (and followers) of the two players. 
 For the Cheating Robot model, this is not completely possible. If Right is the Cheating Robot in $G$ and Left the Cheating Robot in $H$, then in $G+H$ play never starts as each waits for the other.

As Corollary \ref{cor:0unique} showed, there are no negatives in $\SCR$ but we use the negative sign for the conjugate for convenience. In $-G$, the Left and Right moves of $G$ are interchanged and negated, thus the Left and Right options are interchanged and negated, but Right remains the Cheating Robot. The same-round options go from $G^{i,j}$ to $-G^{j,i}$.

\begin{definition}$[\normalfont{\textrm{Conjugate of $G$}}]$
Let $G\in \SCR$, then $-G = \left\{-G^{R}\mid S(-G) \mid -G^{L}\right\}$, where Right is still the Cheating Robot. 
\end{definition}

For example, if $G =\{\ul{1}\mid \ul{2}\mid\ul{-2}\}$ and $H=\{\ul{1}\mid \{\ul{-2}\mid \ul{-1}\mid \ul{3}\}    \mid\ul{5}\}$, then 
$-G = \{\ul{2}\mid \ul{-2}\mid\ul{-1}\}$, 
 and $-H = \{\ul{-5}\mid \{\ul{-3}\mid \ul{1}\mid \ul{2}\}    \mid\ul{-1}\}$.

Although, in general, $G-G$ is not $\ul{0}$, we do know the outcome. The next section shows that $G-G$ can equal $\ul{0}$ if the set of games is restricted.
\begin{lemma}\label{lem: G-G}Let $G\in\SCR$ and $G \neq 0$ then $0\succ G-G$.
\end{lemma}
\begin{proof} For any $X$, if Right can win (draw) $X$, then by playing the corresponding move in the $G$ and $-G$, Right wins (at least draws) $G-G+X$ as well. Therefore $\ocr(X) \geq \ocr(G-G+X)$, which proves $0\succeq G-G$. 

By Corollary \ref{cor:0unique}, $G-G\not= 0$, and the result follows. 
 \end{proof}


\section{Simple Hot Games}\label{sec: simple hot games}

In this section we build, and analyze, a subclass of games in which there is, at most, one significant move remaining in each summand of any disjunctive sum. That is, each summand is either an integer or has one move remaining in which a player can gain an advantage.

 \begin{definition}Let $a,b,c$ be integers with $a\geq b\geq c$. A game $G$, in which the players have exactly one move, Left to $\ul{a}$, Right to $\ul{c}$, and $\ul{b}$ is the same-round option,  is a \textit{simple hot}  game. 
The subclass $\SCR_{SH}$ consists of all disjunctive sums of 
 simple hot games and integers.
 \end{definition}
  This subclass is closed under options and has nicer properties than $\SCR$. In particular, integer games add the same as addition for integers and there are non-trivial games which are equivalent to $0$. Note that in this section we will only be considering games in the subclass, $\SCR_{SH}$, and so we will omit the subscript $SH$ when using $\equiv$ and $\succeq$.  The only change in Definition \ref{def: equality}, the definition of equality and inequalities, is that, now,  $X\in \SCR_{SH}$. Since $\SCR_{SH}$ is a subset of $\SCR$, the conclusions of Lemma \ref{lem:zero} and Theorem \ref{thm:int_ineq}
    still hold. In addition, the next result shows that Lemma~\ref{lem: G-G} holds with the possibility of equality.
\begin{lemma}\label{lem:G-G_2}
Let $G\in\SCR_{SH}$ then $0\succeq G-G$.
\end{lemma}    
\begin{proof}    
Since the first part of the proof of Lemma~\ref{lem: G-G} holds for all $\SCR$ games, it holds for simple hot games.
\end{proof}
    
We now develop the properties of $\SCR_{SH}$ that do not exist in $\SCR$.

\begin{lemma}\label{lem:addition_of_integers} If $a$ and  $c$ are  integers, then  
$\underline{a}+\underline{c}\equiv \underline{a+c}$.
\end{lemma}

\begin{proof} 
If $a$ and $c$ are both non-positive or both non-negative, the result follows from Lemma \ref{lem:zero} and Theorem~\ref{thm:int_ineq}. 
We now suppose that $a > 0 > c$ and we must prove that for all $X$ in $\SCR_{SH}$, $\ocr(\ul{a}+\ul{c}+X) = \ocr(\ul{a+c}+X)$. We will proceed by induction on the birthday of the disjunctive sum of the simple hot games in $X$.

First, suppose $X$ is a disjunctive sum of integers. By Theorem \ref{thm:int_ineq}, we can combine all positive integer components, and all negative integer components respectively, to obtain $X\equiv \ul{p}+\ul{q}$ for some $p\geq 0\geq q$. In both $\ul{a}+\ul{c}$ and $\ul{a+c}$,  each players can only move in their integers; positive for Left, negative for Right. In  $\ul{a}+\ul{p}+\ul{c}+\ul{q}=\ul{a+p}+\ul{c+q}$, Left has exactly  $a+p$ moves and Right $-c-q$ for a difference of $a+p+c+q$. In the second, the difference is $(a+c)+p+q$. Thus $\ocr(\ul{a}+\ul{c}+X)=\ocr(\ul{a+c}+X)$.
 
Suppose that $X$ is the disjunctive sum of integers and simple hot games.
By Theorem \ref{thm:int_ineq},  $X\equiv \ul{p}+\ul{q}+X'$ where $X'$ is the disjunctive sum of simple hot games.
Now $\ocr(\ul{a}+\ul{c}+X) = \ocr(\ul{a+p}+\ul{c+q}+X')$.
Thus, for the rest of the proof, we may now assume that $X$ has no integer summand.

As the first step, we prove  $G=\underline{1}+\underline{-1}\equiv 0$. Let $X\in \SCR_{SH}$. First, we prove $\ocr(X)\geq \ocr(G+X)$. By Lemma ~\ref{lem: G-G}, $0\succ \ul{1}+\ul{-1}$, thus  $o(0+X)\geq o(G+X)$.

Second, we show $\ocr(G+X)\geq \ocr(X)$. If $\ocr(G+X)=\LL$ there is nothing to show, so we may assume that, in $G+X$, Right has a winning (or drawing) response to every Left move. In $X$, Left moves to $X(i,\cdot)$ and Right considers his winning (drawing) response in $\ul{1}+\ul{-1}+X(i,\cdot)$. Suppose 
 a best response is to
$\ul{1}+\ul{-1}+X^{i,j}$. 
In $X$, he responds to $X^{i,j}$ and $\ocr(\ul{1}+\ul{-1}+X^{i,j})\geq \ocr(X^{i,j})$, by induction. Note, if $\ul{1}+X(i,\cdot)$ is a good response (winning or drawing) for Right, then Right playing in the same hot game as Left is even better. This follows since the first pair of moves  yields $\ul{1}+\ul{p}+X'$ for some integer $p$ and the
 second yields $\ul{1}+\ul{-1}+\ul{q}+X'$. We know $p\geq q$ so that 
$\ul{p}\succeq\ul{q}\succ \ul{-1}+\ul{q}$.

We have shown that for $G = \underline{1}+\underline{-1}$,  $o(G+X) = o(X)$ for all $X \in \SCR_{SH}$ and thus $G \equiv 0$.
  
Now suppose that $a\geq 1$ and $-1 \geq c$, then by Theorem \ref{thm:int_ineq}, again, 
\begin{eqnarray*}
\ul{a}+\ul{c} &\equiv& \sum_{i=1}^a\ul{1}+\sum_{j=1}^{-c}\ul{-1}\\
            &\equiv& \ul{1}+\ul{-1} +\sum_{i=1}^{a-1}\ul{1}+\sum_{j=1}^{-c-1}\ul{-1}
\end{eqnarray*}
By the first part of the proof $\ul{1}+\ul{-1}\equiv 0$ thus
\[\ul{1}+\ul{-1} +\sum_{i=1}^{a-1}\ul{1}+\sum_{j=1}^{-c-1}\ul{-1} \equiv \sum_{i=1}^{a-1}\ul{1}+\sum_{j=1}^{-c-1}\ul{-1}.\]
Pairing off the $\ul{1}$ and $\ul{-1}$ summands gives $\ul{a}+\ul{c}\equiv \ul{a+c}$. 

\end{proof}

Even in simple hot games, conjugates are not necessarily inverses. For example, let $G=\{\ul{2} \mid\ul{0}  \mid\ul{-4} \}$ and $H=\{\ul{3} \mid\ul{0}  \mid\ul{-1} \}$. Now $\ocr(G+\ul{0})=\TT$ but Right wins $G+H-H$ by responding in whichever of $G$ and $H-H$ Left didn't move. Thus $H-H\not\equiv 0$.

 The question for the players is how to play in a disjunctive sum of simple hot games and integers.
For example, consider $\{\ul{10}\mid \ul{8} \mid \ul{-3}\} + \{\ul{5}\mid \ul{-2} \mid \ul{-4}\}+\ul{2}+\ul{-1}$.
It seems clear  that Left should play in the first component and Right in the second, and that neither player wants to play in the integers. We prove this in the next  theorem, which is also an \emph{integer translation} result. We will use Lemma \ref{lem:addition_of_integers}, i.e., integers add, without referencing the Lemma each time.

\begin{theorem}\label{lem:integer translation} Let $H=\{\ul{a}\mid \ul{b} \mid \ul{c}\}$ be a simple hot game. If $d$ is an integer then
$H +\ul{d}= \{\ul{a+d}\mid \ul{b+d}\mid \ul{c+d}\}$.
\end{theorem}
\begin{proof}
Note that if a player moves in an integer, the move and the option are identified.\\

 If $d=0$ then there is nothing to prove.

Suppose that $d>0$. Recall that Right has no move in $\ul{d}$. See Figure~\ref{fig: game matrix_G+H_d positive} for the game matrix. We will show that the second row is dominated.

\begin{figure}[ht]
\begin{center}
$
\begin{matrix}
\\
$$H(1,\cdot)+\ul{d}$$\\
\phantom{stuff}\\
$$H+\ul{d-1}$$\\
\phantom{stuff}\\
$$\emptyset$$\\
\end{matrix}
\begin{matrix}
\begin{matrix}
$$\quad\quad H(\cdot,1)+\ul{d}\quad\quad$$&$$\quad \quad\emptyset \quad \quad$$\\
\end{matrix}\\
     \left\llbracket
     \begin{matrix}
$$\quad\ul{b+d}\quad$$&$$\quad\ul{a+d}$$\\
&\\
$$\quad\ul{c+d-1}\quad$$&$$\quad H+\ul{d-1}$$\\
&\\
$$\quad\ul{c+d} \quad$$&$$\quad H+\ul{d}$$\\
\end{matrix}
\right\rrbracket
\end{matrix}$
\caption{The matrix $S^+(H+\ul{d})$, where $d>0$.}\label{fig: game matrix_G+H_d positive}
\end{center}
\end{figure}

By assumption, we have $b\geq c$ and since $d>d-1$, then $\ul{b+d}\succ \ul{c+d-1}$. 

We now prove that   $\ocr(\ul{a+d}+X)\geq  \ocr(H+\ul{d-1}+X)$. If  $\ocr(\ul{a+d}+X)=\LL$ there is nothing to prove.

Suppose Right can win (or draw) $\ul{a+d}+X$ then we claim Right can win (at least draw) $H+\ul{d-1}+X$.

(1) Suppose Left's move is $H+\ul{d-2}+X$. In $\ul{a+d}+X$, Right considers Left's move $\ul{a+d-1}+X$. If Right's best response is in $X$, then
Right plays the same move in $H+\ul{d-2}+X$ and wins (at least draws) by induction. If Right's best response is to $\ul{a+1+d-1} +X=\ul{a+d}+X$,
then, in $H+\ul{d-2}+X$, Right plays to $\ul{c+d-2}+X$, which is less than $\ul{a+d}+X$, and thus Right wins (at least draws). 

(2) Suppose Left's move is $H(1,\cdot)+\ul{d-1}+X$. In $\ul{a+d}+X$, Right considers Left's move $\ul{a+d-1}+X$. If Right's best response is in $X$, then
Right plays the same move in $H(1,\cdot)+\ul{d-1}+X$ and wins by induction. If Right's best response is to $\ul{a+1+d-1}+X$, then in 
$H(1,\cdot)+\ul{d-1}+X$ Right plays to $\ul{b+d-1}+X$, which is less than  $\ul{a+1+d-1}+X$, and so Right wins (at least draws).

(3) Suppose Left's move is $H+\ul{d-1}+X(i,\cdot)$. In $\ul{a+d}+X$, Right considers Left's move $\ul{a+d}+X(i,\cdot)$. If Right's best response is to $\ul{a+1+d}+X^{i,\cdot}$,
then, in $H+\ul{d-1}+X(i,\cdot)$, Right plays to $\ul{c+d-1}+X^{i,\cdot}$ which is less than $\ul{a+1+d}+X^{i,\cdot}$. If Right's best response is in 
$X(i,\cdot)$, then Right plays the same move in $H+\ul{d-1}+X(i,\cdot)$ and wins (at least draws) by induction.

Thus, $\ocr(\ul{a+d}+X)\geq  \ocr(H+\ul{d-1}+X)$.

The second row in Figure~\ref{fig: game matrix_G+H_d positive} is dominated and thus we have that $S^+(H+\ul{d})$ reduces to the matrix in Figure~\ref{fig: game matrix_G+H_dom}.\\

\begin{figure}[h]
\begin{center}
$
\begin{matrix}
\\
$$H(1,\cdot)+\ul{d}$$\\
\phantom{stuff}\\
$$\emptyset$$\\
\end{matrix}
\begin{matrix}
\begin{matrix}
$$\quad H(\cdot,1)+\ul{d}\quad\quad$$&$$\quad\emptyset \quad\quad $$\\
\end{matrix}\\
     \left\llbracket
     \begin{matrix}
$$\quad\quad\ul{b+d}\quad\quad\quad$$&$$\ul{a+d}\quad\quad$$\\
&\\
$$$\quad\quad\ul{c+d}$\quad\quad\quad$$&$$H+\ul{d}\quad\quad$$\\
\end{matrix}
\right\rrbracket
\end{matrix}$
\caption{The simplified matrix $S^+(H+\ul{d})$, where $d>0$.}\label{fig: game matrix_G+H_dom}
\end{center}
\end{figure}

Finally, note that if, in $H+\ul{d}+X$, both players play in a simple hot game in $X$ then the $(\emptyset,\emptyset)$ entry is $H+\ul{d} = \{\ul{a+d}\mid\ul{b+d}\mid\ul{c+d}\}$, by induction.
Therefore, eliminating the second row in $S^+(H+\ul{d})$ gives $S^+(\{\ul{a}+\ul{d}\mid\ul{b}+\ul{d}\mid\ul{c}+\ul{d}\})$. This proves that,
for $d>0$, $H +\ul{d} = \{\ul{a+d}\mid \ul{b+d}\mid \ul{c+d}\}$.
\\

Now suppose $d<0$. See Figure~\ref{fig: G+H_d negative} for the game matrix. We will show that the second column is dominated.
\begin{figure}[ht]
\begin{center}
$
\begin{matrix}
\\
$$H(1,\cdot)+\ul{d}$$\\
\phantom{stuff}\\
$$\emptyset$$\\
\end{matrix}
\begin{matrix}
\begin{matrix}
$$\quad H(\cdot,1)+\ul{d}\quad\quad$$&$$ H+\ul{d+1} \quad\quad $$&$$\quad\emptyset \quad\quad $$\\
\end{matrix}\\
     \left\llbracket
     \begin{matrix}
$$\quad\quad\ul{b+d}\quad\quad\quad$$&$$\ul{a+d+1}\quad\quad$$&$$\ul{a+d}\quad\quad$$\\
&\\
$$$\quad\quad\ul{c+d}$\quad\quad\quad$$&$$H+\ul{d+1}\quad\quad$$&$$H+\ul{d}\quad\quad$$\\
\end{matrix}
\right\rrbracket
\end{matrix}$
\end{center}
\caption{The matrix $S^+(H+\ul{d})$, where $d<0$.}\label{fig: G+H_d negative}
\end{figure}

By assumption, we have $a\geq b$ and since $d<d+1$, then $\ul{b+d}\prec \ul{a+d+1}$. We need to prove that
$\ocr(\ul{c+d} +X)\leq \ocr(H+\ul{d+1}+X)$.

If $\ocr(H+\ul{d+1}+X)=\LL$ then there is nothing to prove. We may then assume that Right has a winning (drawing) response to any Left move in 
$H+\ul{d+1}$.

Let $X\in \SCR_{SH}$.
In $\ul{c+d}+X$, Left has two types of moves.

(1) Left's move is $\ul{c-1+d}+X$. In $H(1,\cdot)+\ul{d+1}+X$ Right has three possible responses but, by induction, Right doesn't play in $\ul{d+1}$ since he has a move in a simple hot game, which is at least as good. 
If a best move is to $\ul{b+d+1}+X$, then, since $b>c-1$ and $d+1>d$, it follows that $\ocr(\ul{c-1+d}+X)\leq \ocr(\ul{b+d+1}+X)$.  
If Right's best response is in $X$, giving $\ul{a+d+1+p}+X'$, for some integer $p$, then playing the same move in $\ul{c-1+d}+X$ yields 
$\ul{c-1+d+p}+X'$. Since $a>c-1$ and $d+1>d$ then $\ocr(\ul{a+d+1+p}+X')\geq \ocr(\ul{c-1+d+p}+X')$. In all cases, Right wins (at least draws).

(2) Left's move is $\ul{c+d}+X(i,\cdot)$. In $H+\ul{d+1}+X(i,\cdot)$, Right has three possible responses but, by induction, Right doesn't play in $\ul{d+1}$ since he has a move in a simple hot game at least as good. Suppose a best move is in $H$, giving the option 
$\ul{c+d+1+p}+X'$, for some integer $p$. In $\ul{c+d}+X(i,\cdot)$, Right plays in $\ul{d}$ giving $\ul{c+d+1+p}+X'$ and the two games are the same.
Suppose a best move is in $X(i,\cdot)$. This gives $H+\ul{d+1+q}+X'$, for some integer $q$, regardless of being in the same or a different hot game.
 Right plays the same move in $\ul{c+d}+X(i,\cdot)$ giving $\ul{c+d+q}+X'$. Now $\ocr(H+\ul{d+1+q}+X')\geq \ocr(\ul{c+d+q}+X')$ by induction.
 In all cases, Right wins (at least draws).

Therefore, the second column is dominated and can be eliminated.

Again, note that if, in $H+\ul{d}+X$, both players play in a simple hot game in $X$ then the $(\emptyset,\emptyset)$ entry is $H+\ul{d} = \{\ul{a+d}\mid\ul{b+d}\mid\ul{c+d}\}$, by induction.

Thus $S^+(H +\ul{d}) = S^+(\{\ul{a+d}\mid \ul{b+d}\mid \ul{c+d}\})$.
\end{proof}
In the proof of  Theorem \ref{lem:integer translation},  the moves in $\ul{d}$ are shown to be dominated. This proves 
the next result.
\begin{corollary}\label{cor:Cor_1}
Let $H=\{\ul{a}\mid \ul{b} \mid \ul{c}\}$ be a simple hot game and $d$ is an integer. 
In $H+\ul{d}$, both players' optimal move is in $H$.
\end{corollary}
Lemma \ref{lem:addition_of_integers} and Theorem \ref{lem:integer translation} allows a \textit{normalization} of simple hot games: the same-round move can be taken to be $\ul{0}$ at the cost of adding an integer to the disjunctive sum.
\begin{corollary}\label{cor:Cor_2}
For $i=1\ldots n$, let $G_i=\{\ul{a_i}\mid \ul{b_i} \mid \ul{c_i}\}$,  and 
$G'_i=\{\ul{a_i-b_i}\mid \ul{0} \mid \ul{c_i-b_i}\}$ be simple hot games. If $H=G_1+G_2+\ldots +G_n$, 
then 
\[H=\left(\ul{\sum{_{i=1}^nb_i}}\right)+ G'_1+G'_2+\ldots +G'_n.
\]
\end{corollary}

Corollaries \ref{cor:Cor_1} and \ref{cor:Cor_2} show that to find the optimal strategies for a disjunctive sum in 
$\SCR_{SH}$ we need only focus on the simple hot games where the same round move is to $\ul{0}$.

Another consequence of these Corollaries, is that the game is essentially over when the position has been reduced to an integer: if the integer is positive then Left wins, if $0$, then it is a draw, and, if it is negative, then Right wins. Once the games in the disjunctive sum have been normalized, the starting integer can be regarded as arbitrary. Consequently, Left's and Right's strategies must maximize, respectively minimize, the sum of the integers obtained when playing the simple hot games. If Right responds in the same game, we say that Right is playing \textit{locally}. For Right, the \textit{local strategy} is to respond in the same simple hot game as Left for all games in the disjunctive sum.

The local strategy is not always optimal. For example, let $H=\{\ul{6}\mid \ul{0} \mid \ul{-55}\}+\{\ul{10}\mid \ul{0} \mid \ul{-36}\}$. %
If Right plays locally, the result is $\ul{0}$. If Right responds in the other summand then the result is either $\ul{6-36}=-\ul{30}$, or $\ul{10-55}=\ul{-45}$, depending on which summand Left plays.

 The optimal strategies must: for Left, give the order in which she should play the simple hot games;
for Right, decide in which simple hot games he should respond. 

First, an observation that helps clarify Left's ordering of the summands in the disjunctive sum.

\begin{observation}\label{obs:index} Let $G_1=\{\ul{a_1}\mid \ul{0} \mid \ul{c_1}\}$ and $G_2=\{\ul{a_2}\mid \ul{0} \mid \ul{c_2}\}$ be simple hot games and suppose $a_1-c_1>a_2-c_2$. Consider $H=G_1+G_2$.
Suppose Right does not play locally.   
Left playing in $G_1$ results in $\ul{a_1+c_2}$, and playing in $G_2$ results in 
$\ul{a_2+c_1}$. However, reordering $a_1-c_1>a_2-c_2$ gives $a_1+c_2>a_2+c_1$.
Left maximizes the integer by playing $G_1$. 
\end{observation}

This observation motivates the next definitions. 

 Let $G_i=\{\ul{a_i}\mid \ul{0} \mid \ul{c_i}\}$, $i=1\ldots n$, be simple hot games and let $G=G_{1}+G_2+\ldots+G_n$. The \textit{standard indexing} of $G$ has $\ul{a_i}-\ul{c_i}\geq \ul{a_{i+1}}-\ul{c_{i+1}}$, $i=1,\ldots,n-1$. If $\ul{a_i}-\ul{c_i}=\ul{a_{i+1}}-\ul{c_{i+1}}$ then the ordering is arbitrary.
The \textit{index-order} strategy for Left is to always play the simple hot game of least index.

 Right may be able to do better than the local strategy and we set up an auxiliary graph to describe the optimization problem that Right needs to solve. 

  Let $D_G$ be a  graph, with loops and weighted edges,  and $V(D_G) = \{1,2,\ldots,n\}$. 
Edge $ij$, $i<j$, exists if
$\ul{a_i}+\ul{c_j}<0$  and $w(ij) = -\ul{a_i}-\ul{c_j}$ is the weight on $ij$. The loops have weight $0$. The \textit{standard optimization problem of $G$} is to find the minimum weight matching on $D_G$. 
Let $M$ be a matching of $D_G$. Let $M'$ be the edges of $M$ together with the loops at vertices not incident to an edge of $M$.
The \textit{matching strategy} from $M$, denoted $\alpha(M)$, is: 
\begin{itemize}
\item Suppose $ij\in M$. If Left plays in $G_i$, then Right plays in $G_j$;
\item If $i$ is not incident with an edge of $M$, then when Left plays in $G_i$, Right responds in $G_i$.
\end{itemize}

Let $G$ be a disjunctive sum of normalized simple hot games. The local strategy corresponds to $M$ being empty. Suppose $ij$, $i<j$, is an edge in $D_G$. If Right plays
  locally in $G_i$ and $G_j$, this produces the integer $0$. If Left plays in $G_i$ and Right responds in $G_j$ the resulting integer is $\ul{a_i}+\ul{c_j}$. Observation \ref{obs:index}  indicates that Left would not want to play in $G_j$ before $G_i$. The weight, $w(ij)=\ul{a_i}+\ul{c_j}$, in $D_G$, is the change that results from Right not playing locally but responding in the other component of $G_i$ and $G_j$.  Consequently, Right will only want to consider negatively weighted edges.
  If $M$ is a matching on $D_G$, then the weights of the edges in $M$ give the change over the local strategy when Right plays the matching strategy from $M$.

\begin{theorem}\label{thm:strategy2} For $i=1,2,\ldots,n$, let $G_i=\{\ul{a_i}\mid \ul{0} \mid \ul{c_i}\}$ be simple hot games with the standard indexing. 
 Let $G=G_1+G_2+\ldots+ G_{n}$. The index-order strategy for Left is an optimal strategy for  playing $G$. For Right, the matching strategy from a minimum weighted matching in $D_G$ is an optimal strategy.
\end{theorem} 

\begin{proof} 
Suppose Left plays the index-order strategy. There is a matching $M$, such that Right playing $\alpha(M)$ is an optimal strategy. If Left plays the games in any other order, Right can still play $\alpha(M)$ and get to the same integer or better. We may assume then that Left plays the index-order strategy.\\

Since Left is playing the index-order strategy,  Right's optimal strategy will be $\alpha(M)$ for some matching $M$ on $D_G$. The resulting integer from playing
$\alpha(M)$ is 
\begin{eqnarray*}
\sum_{ij\in M}(\ul{a_i}+\ul{c_j}) &=& \sum_{ij\in M}w(ij).
\end{eqnarray*}
Since $\sum_{ij\in M}w(ij)$ is the minimum weight of any matching in $D_G$, then $\sum_{ij\in M}w(ij)$ is the least integer that Right can achieve.

\end{proof}

\subsection{Simple hot game example: \textsc{Cheating Robot Toppling Dominoes}}\label{subsec:toppling dominoes}

The rules  for \textsc{cheating robot toppling dominoes} are based on the physics of a line of dominoes. For the underlying combinatorial game, see \cite{FinkNSW2015}. Since playing with general positions is beyond our skill, here, we only consider rows composed of a string of black dominoes followed by a string of white dominoes, which we call TD${}_2$. Dominoes are spaced so that a toppling domino can only topple an adjacent domino. If a row has an empty place then the sub-rows to the left and to the right are regarded as separate rows.  \emph{Moves} are only applied to a single row. Left topples a black domino either to the left or the right. Similarly, Right topples a white domino either to the left or the right. The \emph{effect of the move} is as follows: Two dominoes, $A$ and $B$, and directions are chosen. Domino $A$ topples everything in its direction stopping at $B$, if $B$ is in that direction, else to the end of the string. Similarly for domino $B$. A domino can be `toppled' by the dominoes on both sides at the same time. At the end of the move, all toppled dominoes are removed. The \textit{winner} is the player with dominoes remaining when the opponent has none. 

The next examples illustrate the rules, with no pretence of these being good moves. Let `${}_{-}$' indicate an empty spot and $\overrightarrow{\text{K}}$ or $\overleftarrow{\text{K}}$  indicate the direction in which the player topples the piece `K'. 
Possible moves are
\begin{center}
\includegraphics[width=.25\textwidth]{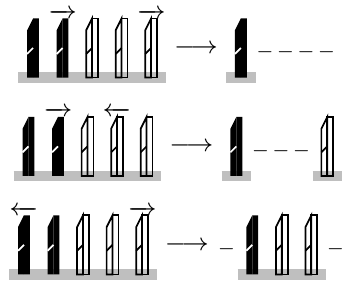}
\end{center}
For brevity, a position with $p$ black dominoes and $q$ white will be written \includegraphics[width=.07\textwidth]{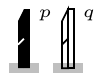} 
and $(p,q)$ in text.

To illustrate the simple hot games strategies, consider the following position $G$:
\[
\includegraphics[width=\textwidth]{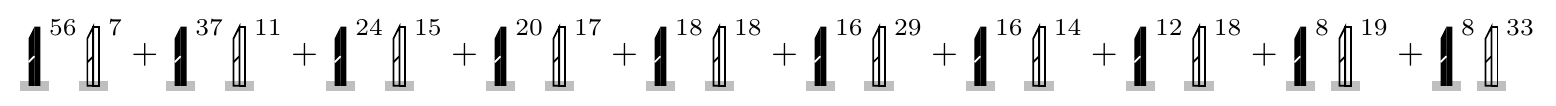}
\]

If you were Left, would you want to play? How would the robot play?\\

For a single row, it is obvious that the best moves are to topple the domino next to the opponent's dominoes and in that direction. The proof we leave to the reader. This shows that TD${}_2$ positions are simple hot games.
\begin{lemma}
The TD${}_2$ position $(p,q)$ is equivalent to $\{\ul{p-1}\mid \ul{p-q} \mid\ul{1-q}\}$ in $\SCR_{SH}$.
\end{lemma}

 Now we can use Theorem \ref{thm:strategy2}. We let $G = (p_1,q_1)+(p_2,q_2)+\ldots +(p_{10},q_{10})$
where they are listed in the index-order strategy for Left. The local-strategy yields an overall score of $34$, which is in Left's favour. The problem is to determine Right's best response.

The normalized form of $\{\ul{p-1}\mid \ul{p-q} \mid\ul{1-q}\}$ is $\ul{p-q} + \{\ul{q-1}\mid \ul{0} \mid\ul{1-p}\}$. In $G$, the sum of the integers is equal to the local strategy (score of $34$).  In the auxiliary graph $D_G$, there is an edge between $(p_i,q_i)$ and $(p_j,q_j)$, $i<j$, if $q_i-1+1-p_j<0$. That is, if $q_i<p_j$. 
These edges are given in
 Figure \ref{fig: Right's matching options} on which Right has to find the optimal weighted matching. It is easy, albeit tedious, using the original or normalized versions of $G$, to check that Right will choose to match only  $(56,7)$ with $(37,11)$,  and $(24,15)$ with $(20,17)$. This yields  an overall score of $-1$, that is, Right wins.

\begin{figure}
\begin{minipage}[t]{0.5\textwidth}
\includegraphics{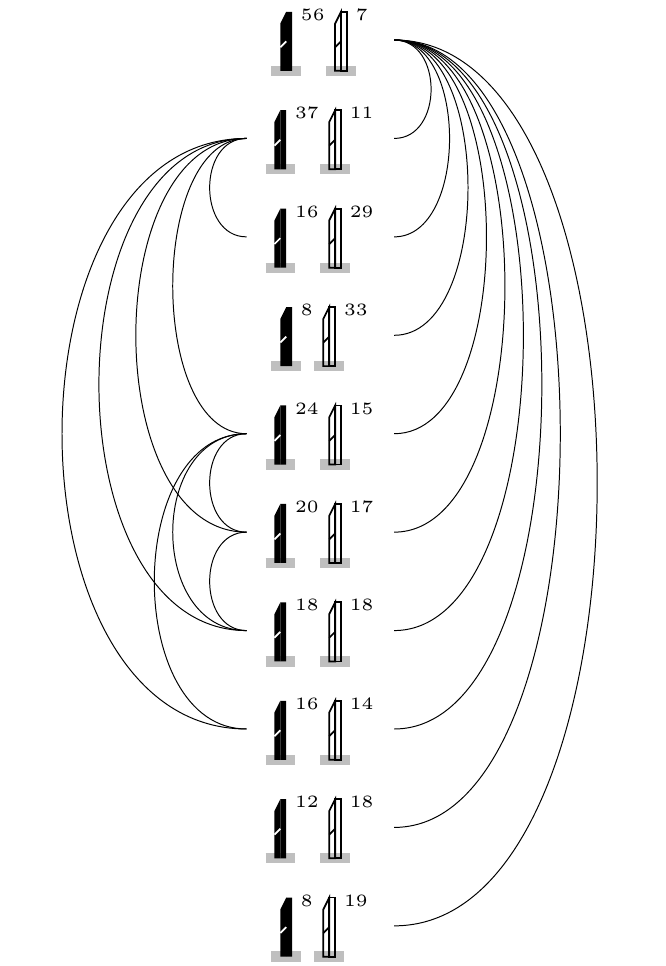}
\end{minipage}
\begin{minipage}[t]{0.5\textwidth}
\includegraphics{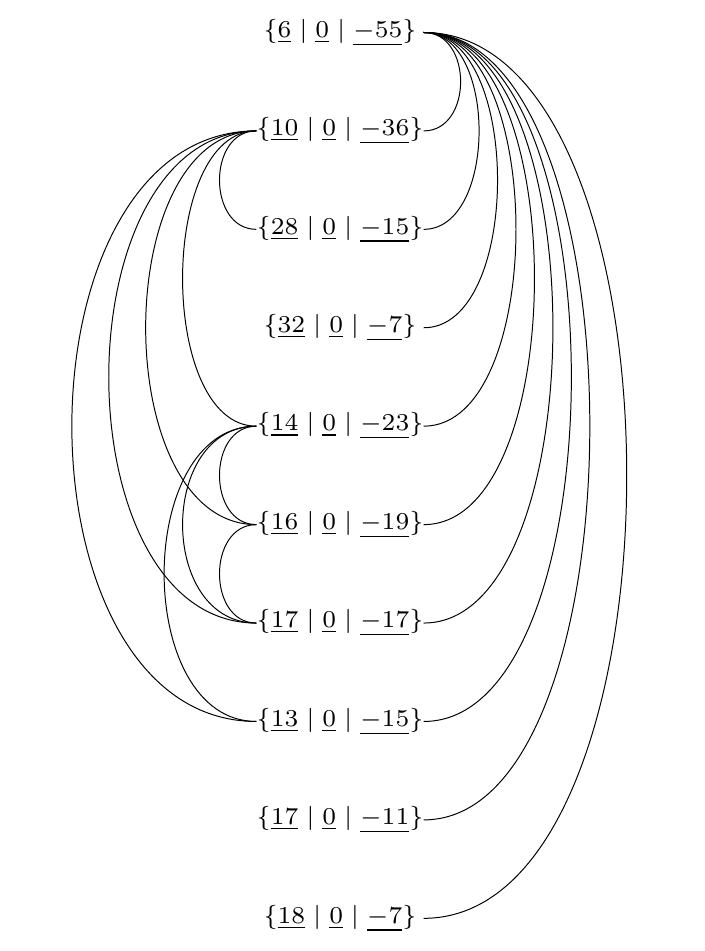}
\end{minipage}
\caption{Options for Right via matchings which are better than the local-strategy; the original summands (left) and the normalized game (right), not including the translated integers.}\label{fig: Right's matching options}

\end{figure}

\begin{question}
The abstract form of the TD${}_2$ position $(p,q)$ is $\{\ul{p-1}\mid \ul{p-q}\mid \ul{1-q}\}$, and the same-round move is directly dependent on $p$ and $q$. Can this special structure 
be used to find a constant time algorithm to determine Right's best response to Left's move in TD${}_2$?
\end{question}


\section{Discussion}
The class $\SCR$ is reminiscent of the class of short mis\`ere games in that (i) the equivalence class of $0$ has one element, and (ii) there are no inverses. However, we showed that there is a sub-class with more algebraic structure. 

As in the other models of combinatorial games, one important question is to find, if possible, a test for $G\succ H$ that involves only 
the followers of $G$ and $H$. If the absolute theory \cite{LNS2021} applied, this would be automatic (see \cite{Siegel2015} and \cite{LNS2016}). However, the absolute theory involves alternating play. We have yet to find a way to adapt this theory to the Cheating Robot context. 

Are there any reductions in addition to those of Theorems \ref{thm:dom2}, \ref{thm:Rightdom2}, and \ref{thm:dom1}? In normal play, there are only two reductions. Repeatedly applying them, in any order, results in a unique game, the canonical form. In fact, the canonical form is the unique game of the least birthday equal to the original. The canonical form of a game $G$, can replace $G$ in any disjunctive sum, thereby reducing the time required to analyze the sum. In the Cheating Robot context, it is not obvious that there is one game of least birthday equal to the original game. In scoring games \cite{LNS2016}, there are three reductions to obtain an equal game of least birthday, however, a substitution result is needed to give a well-defined game that can be called the canonical form.  

Another direction is to find sub-classes in which some or all games have inverses and, in these, find reductions that result in a canonical form. In additive combinatorial game theory, games that have positive or non-negative incentives have the richest algebraic structure, see Milnor \cite{Milno1953} and \cite{AlberNW2019,BerleCG2001,Siege2013} respectively. Following on from Section \ref{sec: simple hot games}, is there a rich theory in the corresponding same-round hot games, that is, where  for every $i$ and $j$,
$G^{i,\cdot} \succeq G^{i,j}\succeq G^{\cdot,j}$?

\section*{Acknowledgements}
We would like to thank Silvia Heubach for her helpful suggestions. 



\begin{thebibliography}{99}
\bibitem{AlberNW2019}Albert MH, Nowakowski RJ, Wolfe D (2019)
\textit{Lessons in Play}, 2nd edition, {Taylor and Francis}.


\bibitem{BahriK2010} Bahri S, Kruskal CP (2010) New Solutions for Synchronized Domineering, in \emph{Proc. of the Conference on Computers and Games 2010}, Kanazawa, pp. 211--229. 

\bibitem{BerleCG2001}
Berlekamp ER, Conway JH,  Guy RK (2001) \textit{Winning ways for your mathematical plays. {V}ol. 1}, A K Peters, Ltd.

\bibitem{CincoI} Cincotti A, Hiroyuki I (2008), The Game of Synchronized Domineering,  \textit{Computers and Games: 6th International Conference, CG 2008, Beijing, China, September 29 - October 1, 2008. Proceedings}, pp: 241--251.

\bibitem{FinkNSW2015}Fink A, Nowakowski RJ, Siegel AN, Wolfe D (2015) Toppling conjectures, in Nowakowski RJ (Ed.), {\it Games of No Chance 4}. MSRI Publications {\bf 63}, pp. 65--76, Cambridge Univ. Press.

\bibitem{Huggan} Huggan MA (2019) Studies in Alternating and Simultaneous Combinatorial Game Theory. PhD Thesis. 

\bibitem{LNS2016}  Larsson U, Neto JP, Nowakowski RJ, Santos CP (2016) Guaranteed scoring games, \textit{Electron. J. Combin.} \textbf{23} Paper 3.27, 29 pp.



\bibitem{LNS2021} Larsson U, Nowakowski RJ, Santos CP (2022) Absolute Combinatorial Game Theory, to appear
in Huntemann S, Larsson U (Eds.), {\it Games of No Chance 6}. 


\bibitem{ItoSYI2016} Ito K, Sueishi T, Yamakawa Y, Ishikawa M (2016) Tracking and Recognition of Human Hand in Dynamic Motion for Janken (rock-paper-scissors) Robot. IEEE International Conference on Automation Science and Engineering (CASE).

\bibitem{Milno1953}
Milnor J (1953) Sums of positional games, \textit{Ann. of Math. Stud.} { \emph Contributions to the Theory of Games,  Kuhn HW, Tucker AW (Eds.), Princeton}, \textbf{2}, pp. {291-301}.


\bibitem{ShortHVS2010} Short E, Hart J, Vu M, Scassellati B (2010) No Fair!! An Interaction with a Cheating Robot. \emph{2010 5th ACM/IEEE International Conference on Human-Robot Interation (HRI)}.

\bibitem{Siege2013}Siegel AN (2013) \textit{ Combinatorial Game Theory}, American Math. Soc.

\bibitem{Siegel2015}
Siegel AN (2015) Mis\`ere canonical forms of partizan games, in Nowakowski RJ (Ed.), \textit{Games of No Chance 4}, MSRI Publications {\bf 63}, pp. 225--239. Cambridge Univ. Press.
\end{thebibliography}
\end{document}